\documentclass[reqno]{amsart}

\input xy
\xyoption{all}
\usepackage{epsfig}
\usepackage{color}
\usepackage{amsthm}
\usepackage{amssymb}
\usepackage{amsmath}
\usepackage{amscd}
\usepackage{amsopn}
\usepackage{url}
\usepackage{hyperref}\hypersetup{colorlinks}

\usepackage{color} %\textcolor{red}{Test}

\definecolor{darkred}{rgb}{1,0,0} %can change the intensity in [0,1]
\definecolor{darkgreen}{rgb}{0,0.8,0}
\definecolor{darkblue}{rgb}{0,0,1}

\hypersetup{colorlinks,
linkcolor=darkblue,
filecolor=darkgreen,
urlcolor=darkred,
citecolor=darkgreen}

\numberwithin{equation}{section}
\newtheorem {theorem}{Theorem}
\numberwithin{theorem}{section}

\newtheorem {lemma}[theorem]{Lemma}

\newtheorem {proposition}[theorem]{Proposition}
\newtheorem {corollary}[theorem]{Corollary}
\theoremstyle{definition}
\newtheorem{definition}[theorem]{Definition}
\theoremstyle{remark}
\newtheorem{remark}[theorem]{Remark}
\newtheorem{example}[theorem]{Example}

\expandafter\chardef\csname pre amssym.def at\endcsname=\the\catcode`\@
\catcode`\@=11
\def\undefine#1{\let#1\undefined}
\def\newsymbol#1#2#3#4#5{\let\next@\relax
 \ifnum#2=\@ne\let\next@\msafam@\else
 \ifnum#2=\tw@\let\next@\msbfam@\fi\fi
 \mathchardef#1="#3\next@#4#5}
\def\mathhexbox@#1#2#3{\relax
 \ifmmode\mathpalette{}{\m@th\mathchar"#1#2#3}%
 \else\leavevmode\hbox{$\m@th\mathchar"#1#2#3$}\fi}
\def\hexnumber@#1{\ifcase#1 0\or 1\or 2\or 3\or 4\or 5\or 6\or 7\or 8\or
 9\or A\or B\or C\or D\or E\or F\fi}

\font\teneufm=eufm10
\font\seveneufm=eufm7
\font\fiveeufm=eufm5
\newfam\eufmfam
\textfont\eufmfam=\teneufm
\scriptfont\eufmfam=\seveneufm
\scriptscriptfont\eufmfam=\fiveeufm

\catcode`\@=\csname pre amssym.def at\endcsname

\numberwithin{equation}{section}%fonts
\font\teneufm=eufm10
\font\seveneufm=eufm7
\font\fiveeufm=eufm5
\newfam\eufmfam
\textfont\eufmfam=\teneufm
\scriptfont\eufmfam=\seveneufm
\scriptscriptfont\eufmfam=\fiveeufm
\DeclareMathAlphabet{\mathpzc}{OT1}{pzc}{m}{it}

\def    \HC     {\operatorname{HC}}
\def    \HH     {\operatorname{H}}

\def    \C      {\operatorname{C}}

\def    \Lnc   {\operatorname{\Lambda^{n-1}_{\comp}}}
\def    \LNC   {\operatorname{\Lambda^{n}_{\comp}}}

\def    \a     {\operatorname{\alpha}}

\def    \H     {\operatorname{\mathcal{H}}}
\def    \A     {\operatorname{\mathcal{A}}}
\def    \s     {\operatorname{\mathfrak{s}}}
\def    \k     {\operatorname{\kappa_{1}}}
\def    \kk    {\operatorname{\kappa_{2}}}

\def    \MUCZ  {\operatorname{\mu_{\scriptscriptstyle{CZ}}}}
\def    \MURS  {\operatorname{\mu_{\scriptscriptstyle{RS}}}}
\def    \DGW   {\operatorname{\Delta_{\textsc{DGW}}}}
\def    \D     {\operatorname{\Delta}}

\def    \eul   {\operatorname{\textsl{e}}}
\def    \rk    {\operatorname{rk}}

\newcommand{\CHI}{\overset{\circ}{\chi}}
\newcommand{\X}{\mathcal{X}}

\def    \ep    {\operatorname{\epsilon}}
\def    \fT    {\operatorname{\textit{f}_{\scriptscriptstyle T}}}

\def    \ker   {\operatorname{ker}}

\def    \V     {\operatorname{V}}

\def    \comp  {\operatorname{\mathbb{C}}}
\def    \real  {\operatorname{\mathbb{R}}}

\def    \ia  {\operatorname{(I1)}}
\def    \ib  {\operatorname{(I2)}}
\def    \ic  {\operatorname{(I3)}}
\def    \id  {\operatorname{(I4)}}
\def    \ie  {\operatorname{(I5)}}

\newcommand{\col}{\colon}

\newcommand{\vs}{\vspace{2mm}}
\begin{document}

\title[Mean Euler Characteristic]{On the Mean Euler Characteristic  of \\ Contact Manifolds}

\author[Jacqui Espina]{Jacqueline Espina}

\address{Universit\'e de Lyon; CNRS;  Universit\'e Lyon 1; Institut Camille Jordan, France}
\email{espina@math.univ-lyon1.fr}

\date{\today} 
\thanks{The work is partially supported by NSF. The research leading to these results has received funding
from the European Community's Seventh Framework Progamme
([FP7/2007-2013] [FP7/2007-2011]) under
grant agreement $\text{n}\textsuperscript{o}$ [258204]}

\begin{abstract} 
We express the mean Euler characteristic of a contact structure in terms of the mean indices of closed Reeb orbits for a broad class of contact manifolds, the so-called \textit{asymptotically finite} contact manifolds. We show that this class is closed under subcritical contact surgery  and examine the behavior of the mean Euler characteristic under such  surgery.  To this end, we revisit the notion of index-positivity for contact forms. We also obtain an expression for  the mean Euler characteristic in the Morse-Bott case.
\end{abstract}

\maketitle

\tableofcontents

\section{Introduction and main results }
\label{sec:intro and results}
\subsection{Introduction} \label{sec:intro}
In this paper, we establish several formulas expressing the mean Euler characteristic (MEC) of a contact manifold via local invariants of periodic orbits of a Reeb flow. True to its name, the mean Euler characteristic  is the average alternating sum of the ranks of cylindrical or linearized contact homology.  
This is a powerful enough invariant to distinguish inequivalent contact structures within the same homotopy class,  such as classes of  Brieskorn manifolds \cite{VK:thesis}, including Ustilosky spheres, \cite{U}. On the other hand as was observed in \cite{GK}, the MEC of a contact manifold can be calculated using only local properties of closed Reeb orbits (the mean indices), provided that the Reeb flow has only a finite number of simple closed orbits. The latter restriction is quite severe and there are very few contact manifolds meeting this requirement.

Our first goal in this paper is to relax the finiteness requirement and replace it by a much more flexible condition which we refer to as \textit{asymptotic finiteness}.  The essence of this condition is that all homological information is carried by a finite collection of orbits shared by a sequence of contact forms, while the indices of the remaining orbits grow and hence their contribution to the contact homology is trivial for any given degree. The contact structures satisfying this condition are quite common and as a consequence of our formula, for each such contact structure, the MEC can be calculated by elementary means bypassing the differential of the contact homology. Furthermore, we show that asymptotic finiteness is preserved by subcritcal surgery, (cf. \cite{U2,W,Yau}),  and that the MEC changes under such surgery in a very simple way. Namely, the difference between the MEC before and after a surgery is $\pm\frac{1}{2}$ for contact manifolds of dimension greater than or equal to 5. This is automatically true in dimension $3$ for the linearized MEC (i.e., the contact connect sum operation changes the MEC by $-\frac{1}{2}$) and we expect to soon show the same is true for the cylindrical version, see Section \ref{sec:dim 3}.  On the other hand, the behavior of contact homology under contact surgery is more involved and is still to be discussed in literature on a rigorous level (cf. \cite{Bo:survey, BEE, BVK}). We also prove a simple formula expressing the mean Euler characteristic for contact forms of Morse-Bott type. 

%An integral part of our analysis is a new notion of index-positivity, called here \textit{weak index-positivity}, relaxing the notion introduced by Ustilosky in \cite{U}, and later investigated by van Koert in \cite{VK:thesis}. Both notions depend on a certain extra structure. For Ustilosky's index-positivity, this is a stable trivialization of the contact structure $\xi$. (Hence $\xi$ must be stably trivial.) For weak index-positivity, this is a section of the determinant line bundle of $\xi$, and hence the only requirement we have is that $c_1(\xi)=0$. The latter condition is much more likely to survive a subcritical surgery and we show that, whenever this is the case, so does weak index-positivity. 

An integral part of our analysis is a new notion of index-positivity, called here \textit{weak index-positivity}, relaxing the notion introduced by Ustilosky in \cite{U}, and later investigated by van Koert in \cite{VK:thesis}.  Both notions depend on a certain extra structure. For Ustilosky's index-positivity, this is a stable trivialization of a contact structure $\xi$. (Hence $\xi$ must be stably trivial.) For weak index-positivity, thinking of $\xi$ as a complex vector bundle, this is a section of the determinant line bundle of $\xi$, and hence the only requirement we have is that $c_1(\xi)=0$. The latter condition is much more likely to survive a subcritical surgery and we show that, whenever this is the case, so does weak index-positivity.

\subsection{Main results} \label{sec:main results}

Let $(M,\xi)$ be a closed contact manifold. Denote the contact homology by $\HC_{*}(M,\xi)$. For the sake of simplicity, we restrict ourselves to cylindrical contact homology of $(M,\xi)$ although our results translate to linearized contact homology in a straightforward fashion; see Section \ref{sec:lch}. We refer the reader to  \cite{Bo:survey, BO, EGH, Ge:contact top} for a general introduction to contact manifolds and contact homology.  Throughout this paper, we assume that $c_1(\xi)=0$.  The contact homology of $M$ breaks down as a direct sum over free homotopy classes of Reeb orbits. Thus the contact homology is defined for any collection of such classes. In what follows, we implicitly assume that this collection is $\{0\}$, that is, we consider only homologically trivial orbits. However,  we can work with all closed Reeb orbits. One should be aware that for homologically non-trivial Reeb orbits, the grading is not unambiguously defined and requires fixing an extra structure on $(M,\xi)$. In our case,  a  convenient way to eliminate this ambiguity is by picking a section of the determinant line bundle of $\xi$; see Section \ref{sec:non-contr}.

Assume that the following condition is satisfied:
\begin{itemize}
	 \item[(CF)]There are integers $l_+$ and $l_-$ such that $\HC_l(M,\xi)$ is finite dimensional for $l \geq l_+$ and $l \leq l_-$.
\end{itemize}
Then we set the \textit{positive/negative mean Euler characteristic} of $(M,\xi)$ to be the following:
\begin{equation} 
\label{eq:pmMEC}
\chi^{\pm}(M,\xi):=\lim_{N \to \infty}\frac{1}{N}\sum_{l=l_{\pm}}^N (-1)^{l}\dim\HC_{\pm l}(M,\xi), 
\end{equation}
provided that the limits exist.  When both limits exist, the \textit{mean Euler characteristic} is set as
\begin{equation}\label{eq:MEC}\chi(M,\xi):= \frac{\chi^{+}(M,\xi)+\chi^{-}(M,\xi)}{2}.\end{equation}
The MEC  was introduced by van Koert  in his investigation of contact structures of a certain class of Brieskorn manifolds; see \cite{VK:thesis}. In a slightly different context, it was also considered by Ekeland-Hofer \cite{EH}, Rademacher \cite{Ra1}, and Viterbo \cite{Vi}.  
As has been pointed out, the mean Euler characteristic can also be defined for various flavors of contact homology as well as  restricted to a specific collection of homotopy classes of closed Reeb orbits. 

The calculation of $\HC_*(M,\xi)$ becomes transcendentally difficult when the Floer differential $\partial$ is nontrivial. On the other hand, similarly to the ordinary Euler characteristic, the mean Euler characteristic can sometimes be calculated in terms of closed Reeb orbits without reference to $\partial$. One such case is considered by Ginzburg and Kerman in~\cite{GK}. 

Namely, it is shown in \cite{GK} that if the following condition is satisfied,
\begin{itemize}
	\item[(CHF)] The Reeb flow has finitely many simple periodic orbits,
\end{itemize}
and  the contact homology is defined, then the MEC  is also defined and  can be expressed as a sum involving only topological indices  and the mean indices of the simple Reeb orbits. The finiteness condition (CHF) is quite restrictive.  Our first goal  is to generalize Ginzburg and Kerman's formula to contact forms with infinitely many simple closed Reeb orbits in a meaningful way.  The first such generalization is to an asymptotically finite manifold mentioned above; for the definition, see Section~\ref{sec:AF}.

\begin{theorem}[MEC formula: asymptotically finite version]
\label{thm:AFMEC} 
Let $(M^{2n-1},\xi)=\{(M,\a_r)\}$ be an asymptotically finite contact manifold and assume that  for each contact form $\alpha_r,$ there are no Reeb orbits of degree $-1,0$ or $1$.
Then the  mean Euler characteristic is defined and
\begin{equation}\label{eq:AFMEC}
\chi^{\pm}(M,\xi)={\sum}^{\pm} \frac{\sigma_{x_i}}{\Delta_{x_i}}+\frac{1}{2}{\sum}^{\pm}\frac{\sigma_{y_i}}{\Delta_{y_i}}, 
\end{equation}
where $\sum ^+$ (respectively $\sum^-$) stands for the sum over the sequences of good principal orbits with positive (respectively negative) asymptotic mean index. 
\end{theorem}
Here, $\sigma_z$ and $\Delta_z$ are the (asymptotic) topological and mean indices of a principal Reeb orbit $z$; see Section \ref{sec:AF}.
We use $x$ and $y$ to distinguish between the two different types of \textit{good} Reeb orbits that generate the contact chain groups. More specifically, we let  $x$ and $y$ denote the Reeb orbits such that the parity of $\MUCZ(x^k)$ remains constant under $k$ iterations while the parity of $\MUCZ(y^k)$ alternates. 

An important feature of the class of asymptotically finite contact manifolds is that this class is essentially closed under subcritical contact surgery.

\begin{theorem}
\label{thm:AFSHA}
Let $(M^{2n-1},\xi)=\{(M,\a_r)\}$ be an asymptotically finite contact manifold. Suppose there is a non-vanishing section $\s$ of $(\Lnc\xi)^{\otimes 2}$ such that 
each $\a_r$ is weakly index-positive with respect to $\s$.  Denote by $(M',\xi')$ a contact manifold obtained from performing a subcritical contact surgery to $(M,\xi)$ of index $k$. If $c_1(\xi')=0$ and if $\s$ extends over the surgery, then $(M',\xi')$ is  asymptotically finite and also weakly index-positive with respect to some extension $\s'$ of~$\s$.
\end{theorem}

\noindent In fact, we prove a more general result (Theorem \ref{thm:WIP}), and as a consequence, we have the following:

\begin{corollary}
\label{cor:AFSHAMEC} 
Let $(M,\xi)=\{(M,\a_r)\}$ and $(M',\xi')$ be  contact manifolds that satisfy the conditions of Theorem \emph{\ref{thm:AFSHA}} and also assume that $n>2$. In addition, for each contact form $\alpha_r$, assume there are no Reeb orbits of degree $-1,0$ or $1$. Then,
\begin{equation}\label{eq:AFSHAMEC}
\chi(M',\xi')=\chi(M,\xi)+(-1)^{k}\frac{1}{2}.
\end{equation}
\end{corollary}

\begin{remark}
\label{remark:lch}
A surgery of index $k=1$ is the contact connect sum. It has been pointed out by Chris Wendl that this operation in dimension $3$ introduces a contractible closed Reeb orbit of degree~$1$. (This simple fact was missed in the first version of this paper.) The condition that there are no contractible closed Reeb orbits of degree $-1,0$, and $1$, is required to define cylindrical contact homology. However, Equation \ref{eq:AFSHAMEC} still holds for the linearized mean Euler characteristic in this dimension, see Section \ref{sec:lch}. At this moment, we are working to prove the same is true for the cylindrical MEC, see Section \ref{sec:dim 3}. 
\end{remark}

\begin{remark}
\label{remark:s'}
As previously mentioned, weak index-positivity is a property of a contact form that is very likely to survive subcritical surgery and now we elaborate on this point. Suppose  $(M',\xi')$ is a contact manifold obtained by performing contact surgery on $(M^{2n-1},\xi)$ of index $k$; see  \cite{W,Yau}. Part of this procedure, very roughly speaking, involves removing a small neighborhood of an embedded isotropic sphere $S^{k-1}_M$ and  glueing a certain contact manifold $\H$ in its place. This contact  manifold  $\H$ can be thought of as  a subset $D^{k}\times S^{2n-k-1}$ of $D^{k}\times D^{2n-k}$ modelled in the standard symplectic space $(\real^{2n},\omega_0)$ and attached to $M$ along $S^{k-1}\times S^{2n-k-1}$, by using the isotropic $(k-1)$-attaching sphere $S^{k-1}\times \{0\}$.  

Assume $c_1 (\xi)=0$. Under contact surgery, except for possibly the case $k=2$,  the first Chern class does not change and a given section $\s$ automatically extends to the new manifold. Suppose $k=2:$ If $S^1$ is contractible, then glueing in a $2$-disc can possibly change the first Chern class. We would need to require $c_1(\xi')=0,$ and  then  $\s$ uniquely extends up to homotopy.  If $ S^1$ is not contractible, then $c_1(\xi')=0$ automatically; however, the section may fail to extend. 
\end{remark}

\begin{remark} 
We again point out that the effect of subcritical surgery on contact homology is still to be discussed in literature on a rigorous level. (However, see \cite{Bo:survey, BVK} for the case of connected sums.) Corollary \ref{cor:AFSHAMEC} gives a simple criterion of when a weakly index-positive contact manifold can be obtained from another via a sequence subcritical surgery. For instance, this cannot be the case for Ustilosky spheres; see Section \ref{sec:examples}.
\end{remark}

Although  a contact manifold of Morse-Bott type with finitely many simple Reeb orbifolds admits a set of  asymptotically finite contact forms, the mean Euler characteristic can be computed without explicitly finding such a set.  This is our next  result, a \textit{Morse-Bott} version of Ginzburg-Kerman's theorem.   In this setting, closed Reeb orbits form  smooth submanifolds.  After taking the quotient of these submanifolds by the natural $S^1$-action induced by the Reeb flow, these orbit spaces are generally orbifolds.   Furthermore, under suitable additional requirements on $(M,\a),$ the Morse-Bott contact homology of $(M,\a)$ is defined and isomorphic to the cylindrical contact homology; see \cite{Bo}. Given in Section \ref{sec:MB-prelims}, these requirements are spelled out in Theorem \ref{thm:MB} which is a main result in \cite{Bo}. With this approach, we prove:

\begin{theorem}[MEC Formula: Morse-Bott version] 
\label{thm:MBMEC}
Let $(M,\a)$ be a contact manifold of Morse-Bott type with finitely many simple Reeb orbifolds satisfying the conditions of Theorem \emph{\ref{thm:MB}}.  Then the mean Euler characteristic of $(M,\a)$ is defined and
\begin{equation}
\label{eq:MBMEC}
\chi^{\pm}(M,\xi)=\underset{\text{max'l} \ S_T}{\sum}^{\pm} \frac{\sigma(S_T)\eul(S_T)}{\Delta(S_T)},
\end{equation} 
where $\sum^+$ (respectively $\sum^-$) stands for the sum over all maximal  orbifolds $S_T$  with positive (respectively negative) mean index.
\end{theorem}

\noindent This time, the MEC formula  involves the indices $\sigma(\cdot)$  and $ \Delta(\cdot)$ defined for Reeb orbifolds. 
We denote by  $\eul(S),$   the following orbifold invariant.  Given an orbifold CW decomposition of a Reeb orbifold $S$, we set
\begin{equation}
\label{eq:orb-eul} 
\eul(S):=\underset{\bar{\sigma}}{\sum}(-1)^{\dim \bar{\sigma}}|Stab(\bar{\sigma})|,
\end{equation} 
where the sum runs over the $q$-cells $\bar{\sigma}$ of the decomposition, and $|Stab(\bar{\sigma})|$ is the order of the  \textit{stabilizer subgroup}.  
In a more general context, $\eul(\cdot)$ is introduced in \cite{ALR} as a variant of the orbifold Euler characteristic naturally arising in orbifold $K$-theory; this is discussed in Section \ref{sec:orb-inv}. The Morse-Bott version of the MEC formula generalizes the relation established in \cite{Ra1} for geodesic flows and also  for closed characteristics on convex hypersurfaces in \cite{HLW}.

\begin{remark} 
In the Morse-Bott case, the MEC formula has a simpler form than in both the asymptotically finite and Ginzburg-Kerman's versions. This is due to the  assumption of Theorem \ref{thm:MB} that there are  no bad Reeb orbifolds.  
\end{remark}

\begin{remark}
\label{remark:index-positivity}  
Suppose $(M,\a)$ is of Morse-Bott type but not all the conditions of Theorem \ref{thm:MB} are satisfied. For instance, if $(M,\a)$ is not required to be index-positive or negative, the Morse-Bott contact homology appears to still be defined although perhaps not necessarily isomorphic to the cylindrical contact homology.  The MEC formula \eqref{eq:MBMEC} still holds, but now the mean Euler characteristic is set via the Morse-Bott contact homology. Index positivity/negativity is necessary to push ``unpredictable'' orbits away to infinite index. %: this also controls index -1,0,1 orbits.
 It has been brought to our attention by van Koert that in \cite{BVK} there are ingredients to remove or relax this assumption. %: however we don't have control over index -1,0,1 orbits. In other words, we only get a linearized version at best.
\end{remark}

\begin{remark}
We again emphasize that homologically non-trivial classes of Reeb orbits can be taken into account and that there are also linearized contact homology versions of these results; see Sections \ref{sec:non-contr} and \ref{sec:lch}.
\end{remark}

\begin{remark}
One can also define the mean Euler characteristic for symplectic homology of a symplectic manifold $(W,\omega)$ with boundary $\partial W$ of contact type.
%One can also define the mean Euler characteristic for symplectic homlogy of a filling $(W,\omega)$ of $(M,\xi)$. 
However, this Euler characteristic is always zero whenever $\chi^{\pm}(W,\omega)$ exists. This follows  from the long exact sequence relating the symplectic and contact homology in \cite{BO}.
On the intuitive level, this is to be expected because every closed Reeb orbit contributes to the chain groups in symplectic homology two generators with degree difference equal to $1$. 
\end{remark}

\begin{remark}
For a symplectic manifold with boundary of contact type as above, another variant of the MEC is given in \cite{FSVK} for the positive part of the equivariant symplectic homology of $(W,\omega)$ (as defined in \cite{BO2,Vi2}). In \cite{FSVK}, there is also an expression for the mean Euler characteristic which generalizes our Morse-Bott version of the MEC formula for circle bundles that we give in Example \ref{ex. circle bundle}.
\end{remark}

\begin{remark}
Finally, we would like to acknowledge that a rigorous construction of contact homology depends on a variety of transversality issues, which are currently being worked out; see \cite{HWZ1, HWZ2, HWZ3} and \cite{CM}, and also \cite{LT} and \cite{FOOO} for the case of Floer homology. We refer the reader to \cite{Bo:survey, BO}  for a detailed discussion of transversailty in this context.
\end{remark}

\subsection{Organization of the paper}
In Section \ref{sec:preliminaries}, we briefly recall the  definition of the mean index and contact homology, and set our conventions and notation.
\textit{Asymptotically finite} contact manifolds are defined  in Section \ref{sec:AF} and we prove  Theorem \ref{thm:AFMEC}  in Section \ref{sec:pf-AFMEC}.    \textit{Weak index-positivity} is introduced in Section \ref{sec:WIP-def}. Contact surgery and handle attaching make the main theme of Section \ref{sec:handles} which also includes the proofs of Theorem \ref{thm:AFSHA} and Corollary \ref{cor:AFSHAMEC}. The variants of the MEC formulas for non-contractible and linearized contact homology are discussed in Section \ref{sec:remarks}.  Section \ref{sec:MB} deals with the Morse-Bott version of the MEC formula and here we prove Theorem \ref{thm:MBMEC}. Finally in  Section \ref{sec:examples}, some examples are given.

\subsection{Acknowledgments}
The author would like to express her gratitude to Viktor Ginzburg for his constant interest, encouragement, advice, and introducing her to symplectic geometry and contact homology.
The author would also like to thank Yasha Eliashberg, Hansj\"org Geiges, Alexandru Oancea, and Otto van Koert for useful discussions. In addition, she thanks  Chris Wendl for pointing out the issue with contact surgery in dimension 3 and also for giving helpful suggestions.

\section{Preliminaries}
\label{sec:preliminaries}

\subsection{Indices of symplectic paths}
\label{sec:indices 1} 
In this section, we set some notation and state some properties for the mean and Conley-Zehnder indices defined for paths of symplectic transformations; see \cite{Long, SZ} for definitions.  Let  $(\real^{2n},\omega_0)$ be the standard symplectic vector space with a compatible complex structure $J_0$ and denote by $Sp(2n)$, the symplectic matrices  represented in the standard symplectic basis.  For a path of symplectic maps $\Psi\col[0,T]\rightarrow Sp(2n)$,   set   $\Delta(\Psi)$ to be its mean index, and if $\Psi$ is non-degenerate (i.e., $\det(\Psi(T)-I)\neq 0$), the Conley-Zehnder index is defined and we denote this by $\MUCZ(\Psi)$.  The mean index is a homogenous quasimorphism, i.e., this index satisfies the following:
\begin{enumerate}\label{items}
  \item[(Hom)]\label{Hom} if $\Psi(0)=I$ and $\Psi(kT+t)=\Psi(t)\Psi(T)^k$ for every $k\in \mathbb{N}$ and $t\geq 0,$ then $\D(\Psi^k)=k\D(\Psi)$ for every 
								$k\in\mathbb{N},$ and 
	\item[(QM)] $|\D(\Psi_1 \Psi_2)-\D(\Psi_1)-\D(\Psi_2)|<c$ where the constant $c$ is independent of the paths. 
\end{enumerate} 
\noindent For non-degenerate paths, we have: 
\begin{itemize}	
 	\item[$\ia$] $|\MUCZ(\Psi)-\D(\Psi)|<n,$ and

	\item[$\ib$] $\lim_{k\to\infty} \frac{\MUCZ(\Psi^k)}{k}=\D(\Psi)$.
	
\end{itemize}

There is a generalization of the Conley-Zehnder index introduced by Robbin and Salamon given in \cite{RS}. This approach provides a useful way to compute  the Conley-Zehnder index and extends the index to degenerate symplectic paths.  In most of this paper, we also call this the Conley-Zehnder index.

\subsection{Cylindrical contact homology}
\label{sec:CCHMEC}
Throughout this paper, $(M^{2n-1},\xi)$ is a closed contact manifold and $c_1(\xi)=0$. Let $\a$ be a contact form for $\xi$, denote the Reeb field by $R_{\a}$, the Reeb flow after time $t$ by $\varphi^t$, and suppose $x\colon[0,T]\rightarrow M$ is a Reeb trajectory. The linearized Reeb flow  preserves the symplectic form $d\a$ giving rise to a family of symplectic maps $d\varphi^{t}_{x}:\xi_{x(0)} \rightarrow \xi_{x(t)}$ along $x$.  Suppose $x$ is a periodic Reeb orbit. In a given trivialization $\Phi_{x}$ of $\xi$ along $x$, set  $\Delta(x):=\Delta(d\varphi^{t}_{x})$, and  if $x$ is non-degenerate (i.e. $\det(I-d\varphi^{t}_{x})\neq 0$), set $\MUCZ(x):=\MUCZ(d\varphi^{t}_{x})$. When $x$ is homologically trivial, there is a canonical way to trivialize $\xi_{x}$ and  both indices are independent of $\Phi_{x}$. However, if $x$ is homologically non-trivial, both $\Delta(x)$ and $\MUCZ(x)$ depend on the trivialization.

Contact homology can be thought of as a Morse theory for the action functional on the loop space of $M$, 
$$\A\col C^{\infty}(S^1,M)\rightarrow \real, \ \gamma\mapsto \int_{\gamma}\a.$$ 
The critical points of $\A$ are the closed orbits of the Reeb flow with period $T=\A(\gamma)$ when parameterized so that $\dot{\gamma}(t)=R_{\a}(\gamma(t))$. For $\A$ to be a Morse functional, the contact form must be chosen generically. This means that all closed Reeb orbits are non-degenerate. When there are homologically non-trivial Reeb orbits,  ambiguity arises in the definition of contact homology  due to the dependency of $\MUCZ(\gamma)$ on $\Phi_{x}$. However, contact homology can be restricted to homotopy classes of the Reeb orbits, so for  simplicity, we may first consider only contractible closed Reeb orbits. This also works for  homologically trivial orbits, and throughout this paper, contractibility can be replaced by homological triviality.  A brief description  of the contact chain complex restricted to the contractible class is given here and a discussion about the homologically non-trivial Reeb orbits can be found in Section~\ref{sec:remarks}.

Suppose $\gamma$ is a simple closed Reeb orbit and denote by $\gamma^k$ its $k^{th}$ iteration. Determined by the way $\MUCZ(\gamma^k)$ behaves,  $\gamma$ can be one of the following two types.  Either, 
\begin{itemize}
	\item[(I)] the parity of $\MUCZ(\gamma^k)$ is the same for all $k\geq 1$, or 
	\item[(II)]the parity for the even multiples $\MUCZ(\gamma^{2k}), k\geq1,$ disagrees  with the parity of the odd multiples $\MUCZ(\gamma^{2k+1}),  
						k\geq 1.$ 
\end{itemize}
All even iterations of type II orbits are called \textit{bad}  and  closed Reeb orbits that are not bad are called \textit{good}.  For any closed Reeb orbit $\gamma$, its degree is set as $|\gamma|:=\MUCZ(\gamma)+n-3$.  The cylindrical contact chain complex  $\C_* (M,\a)$ is the  $\mathbb{Q}$-module freely generated by the good  Reeb orbits graded by $|\cdot|$.   The differential $\partial$ counts holomorphic cylinders between the good orbits of degree difference $1$ in the \textit{symplectization} of $(M,\a)$, however, our results do not explicitly use $\partial$ so we  refer the reader to \cite{Bo:survey} for its definition. If there are no contractible closed Reeb orbits of degree $-1,0,$ and $1$, then cylindrical contact homology is defined and independent of $\a$ as well as the complex structure $J$.   To simplify notation, we write $\HC_{*}(M,\xi)$ for   cylindrical contact homology and use the notation $\HC^{cyl}_*(M,\xi)$ when we compare this with  other variants of contact homology.

\section{Asymptotically finite contact manifolds}

\subsection{Asymptotically finite contact manifolds} 
\label{sec:AF}
In this section,  all periodic Reeb orbits are assumed to be non-degenerate and contractible. See Section \ref{sec:remarks} for a discussion on classes of non-contractible Reeb orbits.

\vs
For a contact form $\a$ for $\xi$, consider the subset $\mathcal{P}^{d}(\alpha):=\{\gamma :  -d \leq |\gamma| \leq d \}$  of closed $R_{\a}$-orbits $\gamma$ and  a topological index of an orbit $\sigma(\gamma):=(-1)^{|\gamma|}$. Suppose there is a sequence of contact forms $\{\alpha _r\}$ for $\xi$ that satisfies the following.
\begin{itemize}

	\item [(AF1)] For each $\alpha_r$, there is a set of simple Reeb orbits $\{\gamma_1(r), 		
				 \ldots, \gamma_m(r)\}$,  where $m$  is independent of $\alpha_r$.
	
	\item [(AF2)] There is an increasing sequence of integers $\{d_r\}$ such that if $\gamma \in \mathcal{P}^{d_r}(\alpha_r)$, then 		
				$\gamma$ is an iteration of some $\gamma_i(r)$,
				i.e., $\gamma=\gamma_i(r)^k$.
	
	\item [(AF3)] The mean indices of each sequence converge:
				$\Delta(\gamma_i(r))\rightarrow \Delta_{i}$.
	
	\item [(AF4)] The sign $\sigma(\gamma_i(r))$ is independent of $r$.
\end{itemize}

\begin{definition}\label{def:AF}
The closed trajectories   $\{\gamma_1(r),\ldots, \gamma_m(r)\}$ of (AF$1$) are called the \textit{principal Reeb orbits} of $\a_r$,  the limits $\{\Delta_1, \cdots, \Delta_m\}$ are called \textit{asymptotic mean indices}, and set $\sigma_i:=\sigma(\gamma_i(r))$ for each $i$.  We say  the contact manifold $(M,\xi)$ is  \textit{asymptotically finite} and   write $(M,\xi)=\{(M,\a_r)\}$. 
\end{definition}

\subsection{Proof of Theorem \ref{thm:AFMEC}}
\label{sec:pf-AFMEC}

We establish the result for $\chi^+$ and the result for $\chi^-$ is proved similarly. Let $(M,\xi)=\{(M,\alpha_r)\}$ be asymptotically finite. Recall from Section \ref{sec:preliminaries} that there are two types of simple closed Reeb orbits. Distinguish the type I and II principal Reeb orbits in (AF1) by $x_i(r)$ and $y_i(r)$ respectively.  Also recall that  
\begin{equation}\label{eq:1} 
|\MUCZ(\gamma^k)-k\Delta(\gamma)|<n-1, 
\end {equation}
and hence
\begin{equation}\label{eq:2}
-2<|\gamma^k|-k\Delta(\gamma)<2n-4
\end {equation}
holds for any periodic orbit of a Reeb flow. 

Let $\{d_r\}$ be the increasing sequence for $(M,\xi)=\{(M,\alpha_r)\}$ satisfying (AF2).   The complex $\C_*(M,\alpha_r)$ is generated by all iterations of type I and odd iterations of type II Reeb orbits of $\alpha_r$. There are integers $l_+$ and $l_-$ such that $\C_l(M,\alpha_r)$ are finite dimensional whenever $l_+\leq l \leq d_r$ and $-d_r\leq l \leq l_-$. The integers can be taken as $l_+=2n-4$ and $l_-=-2$ which follows from \eqref{eq:2} and by the fact that there are finitely many principal Reeb orbits.  We can take $\{\a_r\}$ starting with $r$ such that $d_r\geq2n-4$. Denote by $\C_*^{d_r}(\alpha_r)$ the truncated complex $\C_*(M,\alpha_r)$ from below at $l_+$ and from above at $d_r.$  Set
\[
\C_l^{d_r}(\alpha_r)=
  \begin{cases}
		\C_l(M,\a_r) &\text{if $l_+\leq l \leq d_r$},\\
		0   &\text{otherwise}.
	\end{cases}
\]
The complex $\C_*^{d_r}(\alpha_r)$ is generated by the iterations $x_i(\alpha_r)^k$ and $y_i(\alpha_r)^k$ (for odd $k$) such that $|x_i(\alpha_r)^k|$ and $|y_i(\alpha_r)^k|$ ranges from $l_+$ to $d_r$. By \eqref{eq:2}, this can happen only when $\Delta(x_i(r))>0$ and $\Delta(y_i(r))>0$.  Using \eqref{eq:2} again, we see that orbits with positive mean indices $x_i(\alpha_r)^k$ and $y_i(\alpha_r)^k$ are in $\C_*^{d_r}$ for $k$ ranging between a constant (dependent on the orbit, but independent of $d_r$) up to roughly $d_r/\Delta(x_i(r))$ and $d_r/\Delta(y_i(r))$ up to a constant independent of $d_r$. 
The sequence $\{\alpha_r\}$ has the property that mean indices converge to $\Delta_{x_i}$ and $\Delta_{y_i}$ as $r\rightarrow\infty$. When $R$ is taken large enough, $k$ ranges roughly between some constant to roughly $d_r / \Delta_{x_i}$ and $d_r / \Delta_{y_i}$ whenever $r\geq R$. The parity of $|x_i(r)^k|$ and $|y_i(r)^k|$ are independent of $k$, $\sigma$ is independent of $r$ for each $i$, and so $\sigma(x_i(r)^k)=\sigma(x_i(r))=\sigma_{x_i}$ for all $k$ and $\sigma(y_i(r)^k)=\sigma(y_i(r))=\sigma_{y_i}$ for all $k$ odd.  Then for  large $R$, the contribution to the Euler characteristic 
\[
\chi(\C_*^{d_r}(\alpha_r)):=\sum (-1)^l \dim \C_l^{d_r}(\alpha_r)= \sum^{d_r}_{l=l_+} (-1)^l \dim \C_l^{d_r}(\alpha_r)
\]
by each $x_i(r)$ and $y_i(r)$ is $d_r\cdot\sigma_{x_i}/\Delta_{x_i}+O(1)$ and $(d_r\cdot\sigma_{y_i}/\Delta_{y_i}+O(1))/2$ for all $r \geq R$. 
Then we have 
\[
\chi(\C_*^{d_r}(\alpha_r))= d_r \Big({\sum}^+\frac{\sigma_{x_i}}{\Delta_{x_i}} + \frac{1}{2}{\sum}^+\frac{\sigma_{y_i}}{\Delta_{y_i}}\Big) + O(1), 
\]
and hence
\[
\lim_{r \to \infty} \frac{\chi(\C_*^{d_r}(\alpha_r))}{d_r}=
{\sum}^+\frac{\sigma_{x_i}}{\Delta_{x_i}} + \frac{1}{2}{\sum}^+\frac{\sigma_{y_i}}{\Delta_{y_i}}.
\]

To finish the proof, it remains to show that 
$$\chi^+(M,\xi)=\lim_{r \to \infty} \frac{\chi(\C_*^{d_r}(\alpha_r))}{d_r}.$$
By the definition of $\C_*^{d_r}$, we have $H_l\big(\C_*^{d_r}\big)=\HC_l(M,\xi)$ when $l_+ < l < d_r$. Furthermore,
$|H_{d_r}\big(\C_*^{d_r}\big)-\HC_{d_r}(M,\xi)|=O(1)$ since $\dim \C_{d_r}(M,\alpha_r)=O(1)$.
Hence, 
\[
\chi\big(\C_*^{d_r}\big)=\sum_l(-1)^l\dim H_l\big(\C_*^{d_r}\big)
=\sum_{l=l_+}^{d_r} (-1)^l\dim\HC_l(M,\xi)+O(1)
\]
and the MEC formula follows.  This completes the proof of the theorem.

\section{Weakly index-positive contact manifolds} 
\label{sec:WIP}
In this section, we introduce \textit{weakly index-positive} contact manifolds, a notion similar to Ustilosky's index-positivity where the \textit{unitary index} takes the role of the Conley-Zehnder index. This section includes  the definition of the \textit{unitary index} and a recollection of some useful properties which associates this index to its various relatives.  We also include  a discussion about  index-positivity.

\subsection{The unitary index for symplectic paths}
\label{sec:unitary index}
We give the definition of \textit{unitary index}, discuss some of its properties  and its relations to other indices.

\subsubsection{Definition of the unitary index}
\label{sec:UIdef}
Let $(V,\omega)$ be a symplectic vector bundle over a manifold $M$ of rank $2n$ and fix an $\omega$-compatible complex structure $J$ making $V$ a Hermitian vector bundle.  Denote by $\LNC V$ the top exterior complex power of $V$. Assume $c_1(V)=0$, then $c_1(\LNC V)=0$ and so this line bundle is trivial. Fix a section $\s$ of the unit circle bundle $S^1[(\LNC V)^{\otimes 2}]$ in $(\LNC V)^{\otimes 2}$, and then for $p\in M,$ choose a symplectic basis  of $V_p$  with respect to $\s$. This means to take a symplectic basis $\{e_1, \cdots, e_n, J_0 e_1, \cdots, J_0 e_n\}$ of $V_p$ such that $\{e_1, \cdots, e_n\}$ is a unitary frame that satisfies $(\Lambda^{n}_{j=1}e_j)^{\otimes 2}=\s(p)$. 

Consider a curve $\gamma\col[0,T]\rightarrow M$,  let $V_t$ denote $V_{\gamma(t)}\in V$, and denote by $Sp(V_t)$ the set of symplectic maps of $V_t$. For $t\in[0,T]$, suppose $\Psi(t)\col V_0 \rightarrow V_t$ is a continuous path of symplectic maps over $\gamma$ starting at the identity. Let $\{e_j, J_0 e_j\}^n_{j=1}$  and $\{f_j, J_t f_j\}^n_{j=1}$ be symplectic bases chosen with respect to $\s$ for $V_0$ and $V_t$, respectively. Suppose $A(t)\in Sp(V_t)$ represents $\{\Psi(t)(e_j),\Psi(t)(J_0 e_1)\}^n_{j=1}$ in the  basis $\{f_j, J_t f_j\}^n_{j=1}$. Next, take the polar decomposition $A(t)=P(t)U(t)$, where $P(t)$ and $U(t)$ are the  symmetric, positive-definite and unitary parts, respectively.  Define  $\theta(t)$  by $\det^{2}_{\comp}(U(t))= e^{2i\theta(t)}$. Note that $\theta(t)$ is continuous as it  continuously depends on $A(t)$, and $A(t)$ continuously depends on $\Psi_t$ and $\s$.  Also, note that $\theta(0)=0$  since~$A(0)=I$.

\begin{definition}
Let $\Psi(t)\col V_0\rightarrow V_t$ be a path of symplectic maps over a curve $\gamma\col [0,T]\rightarrow M$ starting at the identity. The \textit{unitary index of $\Psi$ with respect to  $\s$ } is defined by 
$$\mu(\Psi;\s)=\frac{\theta(T)}{\pi},$$ 
where $\theta(T)$ is defined as above. 
\end{definition}

The unitary index is well-defined: For another choice of unitary frames  $\{e'_j\}^{n}_{j=1}$ and $\{f'_j\}^{n}_{j=1}$ for $V_0$ and $V_t$, respectively, there are  change of basis matrices $B_e$ from $\{e_j\}^{n}_{j=1}$ to $\{e'_j\}^{n}_{j=1}$, and $B_f$ from $\{f_j\}^{n}_{j=1}$ to $\{f'_j\}^{n}_{j=1}$. We required 
$$(\Lambda^{n}_{j=1} e'_j )^{\otimes 2} = \s(0) \ \text{ and } \ (\Lambda^{n}_{j=1} f'_j)^{\otimes 2} = \s(t).$$
Both $B_e$ and $B_f$ have complex determinant $\pm1$. If $A'(T)$ is the change of basis between $\{\Psi(t)(e'_j)\}^{n}_{j=1}$ and $\{f'_j\}^{n}_{j=1}$ , then $\det^{2}_{\comp}A'(T)=\det^{2}_{\comp}A(T)$.

\begin{remark}
\label{remark:s}
Although this definition depends on the choice of $\s$, given any homotopic section, say $s'(x)=e^{if(x)}\cdot \s(x)$ where $f\col M \rightarrow \real$, the difference between $\mu(\Psi;\s)$ and $\mu(\Psi;\s')$ is bounded by a constant independent of $\gamma$.  Moreover, if $\gamma$ is a closed path, i.e., $\gamma(0)=\gamma(T)$ and hence $V_T=V_0$, %Note that $\Psi_T:V_0\rightarrow V_0$ need not be the identity.
then $\mu(\Psi;\s)=\mu(\Psi;\s')$. Indeed, we must have $f(\gamma(0))=f(\gamma(T))$ since $\s(\gamma(0))=\s(\gamma(T))$ and $\s'(\gamma(0))=\s'(\gamma(T))$. Also note that when $\pi_1 (M)=0$, all sections of $S^1 [(\Lnc\xi)^{\otimes 2}]$ are homotopic, and so we have the following:
\end{remark}

\begin{proposition} 
If $\pi_1(M)=0,$ then $\mu(\gamma;\s)$ is independent of the section $\s$ for any closed orbit $\gamma$.
\end{proposition}

\subsubsection{Related indices}
\label{sec:indices 2}
The unitary index is a vector bundle version of a quasimorphism defined for symplectic transformations of a symplectic vector space considered by Dupont \cite{Du}, and  Guichardet and Wigner \cite{GW}. We can present this quasimorphism for symplectic maps of the standard symplectic vector space $(\real^{2n},\omega_0)$ with a compatible complex structure $J_0$. Let $\Phi_0 =\{e_1, \cdots, e_n\}$ be the standard unitary frame and take a map $A\in Sp(2n)$. Let  $A=UP$ be the polar decomposition where $U$ is unitary and $P$ is symmetric, positive-definite.  This index is set as  $\DGW(A):=\D(U)$. We recall some properties associated with~$\DGW$.

\vs
\begin{itemize}
	\item[$\ic$] The index $\DGW$ is a quasimorphism: 
							 $$|\DGW(AB)-\DGW(A)-\DGW(B)|\leq K,$$
							 where the constant $K$ depends on $n$; see \cite{Du}.

\vs

	\item[$\id$] The mean index is related to $\DGW$ via the following: 
							 $$|\DGW(A)-\D(A)|\leq C,$$
							 where $C$ is a constant that only depends on $n$; see \cite{BG}. 

\vs

	\item [$\ie$]  Let $\s_0=(\Lambda^{n}_{j=1} e_j)^{\otimes 2}$. 
								 For  a non-degenerate path of symplectic transformations $A(t)\in Sp(2n)$ such that $A(0)=I$, 
								 the following holds:
								 $$|\mu(A;\s_0)-\MUCZ(A;\Phi_0)|\leq C',$$
								 where $C'$ is a constant that only depends on $n$; see below.
\end{itemize}
Recall that for a non-degenerate symplectic path  $A(t) \in Sp(2n)$,  we also have properties $\ia$ and $\ib$ relating $\MUCZ(\Psi)$ and $\Delta(\Psi)$ given in Section \ref{sec:preliminaries}. 
By $\ia, \id$, and  the observation that  $\mu(A;\s_0)=\DGW(A)$, the relation $\ie$  follows.  

We also mention a version of $\ie$ holds for degenerate paths of symplectic transformations. In that case, replace $\MUCZ$ with the generalized Conley-Zehnder index. %takes the place of $\MUCZ$. 

\subsubsection{The unitary catenation lemma}
\label{sec:cat lemma}

Again consider the setting of Section \ref{sec:UIdef}: Let $V$ be a symplectic vector bundle  over a manifold $M$ with $c_1(V)=0$, fix a section $\s$ of $S^1[(\LNC V)^{\otimes2}]$. Let $\gamma\col[0,T]\rightarrow M$ be a curve on $M$ and  $\Psi(t)\col V_0 \rightarrow V_t$ be a path of symplectic maps over $\gamma$ starting at the identity. For $0<T_1<T$, we have 
$$\mu(\Psi,\s)=\mu(\Psi|_{[0,T_1]};\s)+\mu(\Psi|_{[T_1,T]};\s).$$

On the otherhand, suppose for some $T_1\in[0,T]$, we have paths of symplectic maps $\Psi_1(t)$ and $\Psi_2(t)$ over $\gamma_1=\gamma|_{[0,T_1]}$ and $ \gamma_2=\gamma|_{[T_1,T]}$, respectively. If both $\Psi_1(t)$ and $ \Psi_2(t)$  start at the identity,  their \textit{catenation} is defined as
\[
(\Psi_1 * \Psi_2)(t) = \begin{cases}	\Psi_1(t), & 0 \leq t \leq T_1,\\
																	\Psi_2(t)\Psi_1(T_1), & T_1 \leq t \leq T.
												 \end{cases}
\]

\begin{lemma}
[$\mu$-Catenation lemma]
For the paths of symplectic maps $\Psi_1$ and $\Psi_2$ as above, we have
$$|\mu(\Psi_1 * \Psi_2;\s) - \mu(\Psi_1;\s)-\mu(\Psi_2;\s)|\leq b,$$ where $b$ is a constant that depends on the rank of $V$.
\end{lemma}

\begin{proof}
Set $\Psi(t):=\Psi_2(t)\circ\Psi_1(T_1):V_0 \rightarrow V_t$ and suppose $A(t)$ are symplectic matrices representing $\Psi(t)$ with respect to $\s$. Note that $$\mu(\Psi_1 * \Psi_2;\s)=\DGW(A(T))=\mu(\Psi;\s).$$ 
The lemma holds  by the following claim:
$$|\mu(\Psi;\s)-\mu(\Psi_1;\s)-\mu(\Psi_2;\s)|\leq b,$$ 
where the constant $b$ depends on the rank of $V$.

Fix a Hermitian basis $\{e_j\}^{n}_{j=1}, \{e'_j\}^{n}_{j=1}, \{f_j\}^{n}_{j=1}$ for $V_0, V_{T_1}, V_T$ with respect to~$\s$. In this fixed basis, let $A_1(t), A_2(t), A(t)\in Sp(V_t)$  represent $\Psi_1(t)$, $\Psi_2(t)$, $\Psi(t)$, and denote their unitary parts by $U_1(t), U_2(t), U(t)$, respectively. Let $\theta_{1}(t),\theta_{2}(t),\theta(t)$ be  maps  defined by $\det^{2}_{\comp}(U_1(t))=e^{i\theta_1(t)}$, $\det^{2}_{\comp}(U_2(t))=e^{i\theta_2(t)}$, $\det^{2}_{\comp}(U(t))=e^{i\theta(t)}$, respectively.  Then, 
$$ \mu(\Psi_1(t);\s)=\DGW(A_1(T_1))=\frac{\theta_1(T_1)}{\pi},\ \mu(\Psi_2(t);\s)=\DGW(A_2(T))=\frac{\theta_2(T)}{\pi},$$ 
$$ \text{ and } \ \mu(\Psi(t);\s)=\DGW(A(T))=\frac{\theta(T)}{\pi}.$$

\noindent Denote by $B(t)$, the change of basis from 
$\{\Psi_2(t)(e'_j),\Psi_2(t)(J_{t} e'_j)\}^{n}_{j=1}$ to $\{\Psi_2(t)\circ \Psi_1(T_1)(e_j),\Psi_2(t)\circ \Psi_1(T_1)(J_0 e_j)\}^{n}_{j=1}$. 
Then $A(t)\circ B(t)=A_2(t)$.

\vspace{3mm}
\[
\xymatrix@R=7ex{
\txt{$V_0$\\$\{e_j\}$, \\$(\Lambda^n_{\comp} e_j)^{\otimes 2} = \s(\gamma(0))$}
\ar@<1.5ex>[r]%^{\gamma(0)}
\ar[dr]_{\Psi_1} 
&
\txt{$V_{T_1}$\\$\{e_j^\prime\}$, \\$(\Lambda^n_{\comp} e_j^\prime)^{\otimes 2} = \s(\gamma(T_1))$}
\ar@<1.5ex>[r]%^{\gamma(T_1)}
\ar[dr]_{\Psi_2}
\ar[d]_{A_1(T_1)} 
&
\txt{$V_T$\\$\{f_j\}$, \\$(\Lambda^n_{\comp} f_j)^{\otimes 2} = \s(\gamma(T))$}
\ar[d]_{A_2(T)}
\ar@/^3pc/[dd]^{A(T)}
\\
&
\{\Psi_1(T_1)(e_j)\}
\ar[dr]_{\Psi_2} 
&
\{\Psi_2(T)(e_j)\}
\ar[d]_{B(T)}
\\
&
&
\{\Psi_2(T)(\Psi_1(T_1)(e_j))\}
}
\]
\vspace{3mm}

The quasimorphism property $\ic$ of $\DGW$   holds for symplectic maps on any symplectic vector space. In particular, we apply $\ic$ for maps in $Sp(V_T)$, and so $$|\DGW(A(T))-\DGW(A_1(T_1))-\DGW(B(T))|\leq b_{V_T},$$ where $b_{V_T}$ is a constant that depends only on the rank of $Sp(V_T)$. Then since $\DGW(B(T))=\DGW(A_1(T_1))$ and the rank of $Sp(V_p)$ is the same for all $p\in M$, we have  $$|\mu(\Psi;\s)-\mu(\Psi_1;\s)-\mu(\Psi_2;\s)|\leq b,$$ where $b$ depends on the rank of $V$.

\end{proof}

\eject 
\subsection{Weakly index-positive contact manifolds}
In this section, we define \textit{weak index-positive} contact manifolds and discuss Ustilosky's  \textit{index-positivity}.

\subsubsection{Weak-index positivity}
\label{sec:WIP-def}
Let $(M,\xi)$ be a contact manifold  with $c_1(\xi)=0$  and fix a section $\s$ of $S^1[(\Lnc\xi)^{\otimes 2}]$.  Let $\a$ be a contact form for $\xi$, 
denote  the Reeb vector field by $R_{\a}$, let $x\col[0,T]\rightarrow M$ be a Reeb trajectory, and set $\Psi(t):=d\varphi^t_x\col\xi_{x(0)}\rightarrow \xi_{x(t)},$ for $t\in[0,T]$.  Suppose $A(t)\in Sp(\xi_{x(t)})$ represents $\Psi(t)$ with respect to $\s$, for each $t\in[0,T]$, and denote the unitary part of $A(t)$ by $U(t)$. Define  the map $\theta(t)$  by $\det^{2}_{\comp}(U(t))=e^{i\theta(t)}$. 

\begin{definition} 
The \textit{unitary index} of a Reeb trajectory $x\col[0,T]\rightarrow M$  with respect to $\s$ is defined to as
 $$\mu(x;\s):=\mu(\Psi;\s):=\frac{\theta(T)}{\pi}.$$
\end{definition}

\begin{definition}
A contact form $\a$ for $(M,\xi)$  is said to be \textit{weakly index-positive with respect to $\s$} if there exists constants $\k>0,\kk$ such that $$\mu(\gamma;\s)\geq \k\A(\gamma) + \kk,$$ for any Reeb trajectory $\gamma$ with action $\A(\gamma)$.  We call $(M,\xi)$\textit{ weakly index-positive} if $\xi=\ker(\a)$,  for some  contact form $\a$ that is  weakly index-positive with respect to some section $\s$.  (\textit{Weak index-negativity}  can be defined similarly.)
\end{definition}

\begin{remark} 
The unitary index depends on the choice of a section of $S^1[(\Lnc(\xi)^{\otimes 2}]$ and therefore weak index-positivity also depends on this choice.  However,  for any two homotopic sections $\s$ and $\s'$, Remark \ref{remark:s} gives us  $|\mu(x;\s)-\mu(x;\s')|\leq c$  for some  constant $c$ that depends only on the sections. Suppose $(M,\a)$ is weakly index-positive with respect to some  $\s$. When $\pi_1(M)=0,$ all sections are homotopic, and here $(M,\a)$  is weakly index-positive for any $\s$. If $\pi_1(M)\neq0$, we just need to specify the homotopy class of $\s$.
\end{remark}

\noindent Weak index-positivity is also stable under small perturbations. 
\begin{lemma}
\label{lemma:WIP-pert}
If $(M,\xi\in ker(\a))$ is weakly index-positive with respect to $\s$, then for a $C^{\infty}$-small perturbation $\a'$ of $\a$, the manifold $(M,\a')$ is also weakly index-positve with respect to $\s$.
\end{lemma}

\begin{proof}
Let $\k>0, \kk$ be  constants such that $\mu(x;\s)>\k\A(x)+\kk$, for any $R_{\a}$-trajectory $x$. Fix $T>0$ and let $\gamma, \gamma'$ be trajectories of $R_{\a}, R_{\a'}$, respectively, such that $\gamma(0)=\gamma'(0)$ and $\A(\gamma)=\A(\gamma')<T$. Denote by $\Psi, \Psi'$ the linearized Reeb flows of $R_{\a}, R_{\a'}$ along $\gamma, \gamma'$, respectively. Let $\tilde{\gamma}(t):=\gamma(T-t)$ be the inverse path of $\gamma$. Then the inverse path of $\Psi$ is $\tilde{\Psi}(t):=\Psi(T-t)\Psi(t)^{-1}$ over $\tilde{\gamma}$. Consider the path composition of  $\tilde{\gamma}$  and $\gamma'$:

\[
\Gamma(t):= \begin{cases}	\tilde{\gamma}(t),  & 0 \leq t \leq T,\\
														\gamma'(t-T),     & T \leq t \leq 2T.
						\end{cases}
\]
We obtain a symplectic path $\Psi_{\Gamma}$ over $\Gamma$ by taking the catenation:
\[
\Psi_{\Gamma}(t):= (\tilde{\Psi} * \Psi' )(t) = \begin{cases} \tilde{\Psi}(t),       & 0 \leq t \leq T,\\
																														\Psi'(t-T)\Psi(T)^{-1},  & T \leq t \leq 2T.
																							  \end{cases}
\] 
Also, consider the following symplectic maps. For each $0 < s < T$, restrict  $\Gamma$ to $t\in[s,2T-s]$, and set
$$\Psi_s:= \Psi'(s)\Psi(s)^{-1}\col\xi_{\gamma(s)}\rightarrow \xi_{\gamma'(s)}.$$ 
Note that $\Psi_T = \Psi_{\Gamma}(2T)$, and so $\mu(\Gamma;\s)=\DGW(\Psi_T)$. 

Let $A_{\Psi}(t)$ and $ A_{\Psi_s}$ be  symplectic matrices representing $\Psi_{\Gamma}(t)$ and $\Psi_s$, respectively, expressed in  bases chosen by $\s$. If the perturbation is small enough, we have $|\mu(\Gamma;\s)|\leq 2n$. To see this, consider the polar decompostion  $A_{\Psi_s}=P_{\Psi_s}U_{\Psi_s}$. Let $\textsc{I}_p$ be the identity at $p\in M$ with respect to $\s_p$. When $s=0$, the unitary part is $U_{\Psi_0}=\textsc{I}_{\gamma'(0)}=\textsc{I}_{\gamma(0)}$, and for each $s\in [0,T]$, the matrix $P_{\Psi_s}$ is symmetric, positive-definite. For small perturbations, we have 
$$||\textsc{I}_{\gamma'(s)}- A _{\Psi_s}||<1 \ \text{for each } s\in[0,T].$$ 
The $n$ complex eigenvalues of $U_{\Psi_s}$ must then stay between $e^{-i\pi}$ and $e^{i\pi}$ for  $s\in[0,T]$,  and so $-2n<\mu(\Gamma;\s)<2n$.

By the $\mu$-catenation lemma and the fact that $\mu(\tilde{\gamma};\s)= -\mu(\gamma;\s)$, we get
$$|\mu(\gamma;\s)-\mu(\gamma';\s)|\leq 2n + k,$$ 
for a constant $k$ depending on $n$. Then 
$$\mu(\gamma';\s) \geq \mu(\gamma;\s)-2n - k \geq \k\A(\gamma)+\kk - 2n - k.$$

Next, suppose  $\gamma'$ is a trajectory of $R_{\a'}$ with action greater than $T$.  We can break $\gamma'$ into $p$ pieces, $\gamma'_1,\ldots, \gamma'_p$, such that each $\A(\gamma'_i)<T$ and $p-1 < \A(\gamma)/T.$ Again by the $\mu$-catenation lemma,

\begin{align*}
\mu(\gamma';\s) & \geq \sum_{i=1}^{p} \mu(\gamma'_i)-k(p-1)\\
								& \geq \sum_{i=1}^{p}(\k\A(\gamma'_i)+\kk-2n-k)-k(p-1)\\
								& \geq \k\A(\gamma')-|\kk|p - 2np -2pk -k\\
								& \geq \k\A(\gamma')-(|\kk|+2n+2k)\big(\frac{\A(\gamma')}{T}+1 \big)+k\\
								& \geq \big(\k-\frac{|\kk|+2n+2k}{T}\big)\A(\gamma')-|\kk|-2n-k.\\
\end{align*}
All  constants  in the last line are independent of $\a'$ so we can choose any $T>(|\kk|+2n+2k)/\k$. We can take 
$T=2(|\kk|+2n+2k)/\k$ and set $\k'=\k/2$ and $\kk'=-|\kk|-2n-k$. 

\end{proof}

\begin{remark} 
For a small perturbation, there is a one-to-one correspondence
between non-degenerate Reeb orbits of the contact form and of the perturbed
contact form up to some period.  The corresponding
orbits have the same Conley-Zehnder indices.
\end{remark}

\subsubsection{Index-positivity} 
\label{sec:IP}
In the setting of \cite{U}, Ustilosky assumes that $\pi_1(M)=0$, the bundle $(\xi,d\a)$ is stably trivial, and all closed Reeb orbits are nondegenerate. For a stable trivialization $\Phi\colon\xi\oplus\epsilon^2\rightarrow\epsilon^{2n}$, there is a way to define the (generalized) Conley-Zehnder index $\MUCZ(\gamma;\Phi)$ of any Reeb trajectory $\gamma$, not necessarily closed nor non-degenerate.  If $\gamma$ is closed, then $\MUCZ(\gamma;\Phi)$ is equal to the usual  Conley-Zehnder index  and is independent of $\Phi$. For all other trajectories, this index does generally depend on the stable trivialization. However,  given any other stable trivialization $\Phi'$, it is shown that $|\MUCZ(\gamma;\Phi)-\MUCZ(\gamma;\Phi')|$ is bounded by some  constant independent of $\gamma$. With this in mind, the following condition  on  contact forms is independent of $\Phi$.

\begin{definition}
The manifold $(M,\a)$ is called \textit{index-positive} if there exists constants $\kappa_1>0, \kappa_2$ such that $$\MUCZ(\gamma;\Phi)\geq\kappa_1\A(\gamma)+\kappa_2,$$ 
for any  Reeb trajectory $\gamma$. (\textit{Index-negativity} can be similarly defined.)
\end{definition}

\begin{remark} 
\label{rmk:IP}
An index-positive contact manifold $(M,\alpha)$ is also weakly index-positive: Let $\s_{\Phi}$ be the section of $S^1[(\Lnc\xi)^{\otimes 2}]$ obtained by  $\Phi$. Given by $\ie$ from Section  \ref{sec:indices 2}, $|\mu(\gamma;\s_{\Phi})-\MUCZ(\gamma;\Phi)|\leq C'$, for some constant $C'$ independent of $\gamma$.
\end{remark}

\section{Handle attaching} 
\label{sec:handles}

\subsection{Contact surgery}
\label{sec:surgery}
Theorem \ref{thm:AFSHA} concerns the effect of a subcritical surgery on a contact manifold  and the mean Euler characteristic. Very roughly speaking, we glue to $(M^{2n-1},\xi)$ a component $D^k \times S^{2n-k-1}$ of the contact boundary of a symplectic $k$-handle $D^k \times D^{2n-k}$ modelled in $\real^{2n}$. Similar to surgery, this  hypersurface is  determined by a Morse function of index $k$; see for instance \cite{Ge:contact top, W}. \textit{Subcritical} here means that $k<n-1$.  The specific model  used in this paper is constructed in \cite{Yau}, where  the handle is presented with a focus on the dynamics of the Reeb orbits. We give only a rough  description of this model here, some contact boundary properties of these handles are included in the next section, and the reader is referred to \cite{Yau} for complete details.

Consider the standard symplectic space $(\real^{2n},\omega_0)$,  and for $0\leq k\leq n$, take the decomposition $\real^{2n}=\real^k \times \real^k \times \real^{2n-2k}$. Write $(x,y,z)$ for the coordinates  $x=(x_1,\ldots,x_k),$ $ y=(y_1,\ldots,y_k),$ and $ z=(z_{k+1},\ldots,z_n),$ where $z_l=x_l + iy_l,$ for $l=k+1,\ldots,n$. Fix $k$ and set 
$$f(x,y,z)=|x|^2 -\frac{1}{2}|y|^2 +\frac{1}{4}|z|^2.$$
Note that at the origin $\textbf{0}\in \real^{2n}$ is the only critical point of $f$ and its Morse index is $k$. The gradient vector field of $f$ with respect to the Euclidean metric, $Y:=\nabla f$, defines a Liouville vector field on   $(\real^{2n},\omega_0)$.  Set $\a:=\iota_Y \omega_0$;  this 1-form restricts to  a contact form on any hypersurface of $(\real^{2n},\omega_0)$ transverse to $Y$.

Now consider a similar Morse function  
\begin{equation}
\label{eq:handle}
h(x,y,z)=b|x|^2 -b'|y|^2 + \sum^{n}_{l=k+1}\frac{|z_l|^{2}}{c_l}, \text{ \ for positive constants \ } b>b',c_l. 
\end{equation}
A symplectic handle $\tilde{\H}$ is  the locus of points $(x,y,z)\in \real^k \times \real^k \times \real^{2n-2k}$ satisfying 
$$c_{-}\leq h(x,y,z) \leq c_{+},\ \text{\ for some constants \ }\ c_-<0 \ \text{ and }  \ c_+>0.$$
Topologically, this is $D^k \times D^{2n-k}$ and $(S^{k-1}\times D^{2n-k}) \cup (D^k\times S^{2n-k-1})$ is its boundary.

The symplectic handle is chosen such that its boundary  is transverse to $Y$, and so  $\a$  defines a contact form on $\partial\tilde{\H}$.  The \textit{lower boundary} $\{h=c_-\}$, which corresponds to $S^{k-1}\times D^{2n-k}$,  contains the isotropic $(k-1)$-sphere $S^{k-1}_{-}:=S^{k-1}\times ~\{0\}$. When performing contact surgery,  $S^{k-1}_{-}$ serves as the attaching sphere.  The  \textit{upper boundary} $\H:=\{h=c_{+}\}$, which corresponds to $ D^k\times S^{2n-k-1}$, will be the contact boundary component we glue to a contact manifold during surgery.  Denote by $\a_{\H}$, the restriction of $\a$ to $\H$. All closed Reeb orbits of $(\H,\a_{\H})$ are contained in the contact ellipsoid  $S^{2n-2k}_{*}:=\{0\}\times\{0\} \times S^{2n-2k}$, which is also contained in the coisotropic ellipsoid  $S^{2n-k-1}_{+}:=\{0\}\times S^{2n-k-1}$, called the \textit{belt} of the symplectic handle. 

We have freedom in the choice of shape of our contact handles which is determined by the coefficients $b, b'$ and $c_l$, of the handle's defining function $h$. This is due to the fact that the Liouville flow produces a contactomorphism  from the upper and lower boundaries of $\tilde{\H}$, to the upper and lower boundaries of  any symplectic handle of some other shape (with boundary transversal to $Y$).  This flow also preserves the submanifolds $S^{k-1}_{-}, S^{2n-k-1}_{+}$, and $S^{2n-2k}_{*}$. 

Now suppose we want to perform a contact surgery of index $k$ to  a contact manifold $(M,\xi=\ker(\a))$. First, recall just the topological aspect of surgery; the procedure is performed along an (embedded) $(k-1)$-sphere $S^{k-1}_{M}$ of $M$ whose normal bundle is trivial.  This means that a tubular neighborhood of the sphere  $\mathcal{N}(S^{k-1}_M) \cong~S^{k-1}\times D^{2n-k}$ is cut out of $M$ and $\H$ is glued in. Let $M'$ denote the resulting manifold. Next, to guarantee that a surgery induces a contact structure on $M'$, we use the Liouville flows of the symplectic handle and the symplectization of $(M,\a_{M})$; this is where the isotropic requirement of the attaching spheres steps in. Suppose $M$ has an (embedded) isotropic $(k-1)$-sphere $S^{k-1}_{M}$. For an isotropic sphere $S$ of either $M$ or $\H$, the normal bundle of $S$  breaks down  into a direct sum $\left\langle -R\right\rangle\oplus J(TS)\oplus \textsc{CSN}(S),$ where $R$ denotes either $R_{\a}$ or $R_{\a_{\H}}$, and $\textsc{CSN}(S)$ stands for the \textit{conformal symplectic normal} bundle of $S$. This bundle is $\textsc{CSN}(S):=(TS^{k-1})^{\omega}/TS^{k-1}$, where $(TS^{k-1})^{\omega}$ is the $\omega$-complement of $TS$, for $\omega=d\a$ or $d\a_{\H}$. Now since $\left\langle -R\right\rangle\oplus J(TS)$ has a natural framing,  we are left to consider  $\textsc{CSN}(S^{k-1}_{M})$ and $\textsc{CSN}(S^{k-1}_{-})$. An arbitrary framing can be chosen for $\textsc{CSN}(S^{k-1}_{-})$. When $\textsc{CSN}(S^{k-1}_{M})$ is trivial, a framing of $\textsc{CSN}(S^k_M)$  allows us to make the  identification of neighborhoods of the isotropic spheres (via the Liouville flow $Y$). The Liouville vector field is transverse to $M'$ and thus this new manifold  indeed carries a contact structure as a result of this surgery.

\subsection{The contact boundary of a subcritical symplectic handle}
\label{sec:handle props}

We are interested in the effect of a subcritical surgery on a contact manifold. This means that the index of the surgery $k$ is strictly less than $n$, where $2n-1$ is the dimension of the contact manifold. Recall that a component of the contact boundary of a symplectic $k$-handle will  be glued to $(M,\xi)$ to obtain a new contact manifold; see Section \ref{sec:surgery}. We remind the reader that the specific handle model used here is constructed in \cite{Yau}. Let $(\H,\xi_{\H}=\ker(\a_{\H}))$ denote this contact boundary component and denote its Reeb field by $R_{\H}$.

%\eject
\begin{theorem}[Theorem 3.1 of \cite{Yau}] 
\label{thm:Yau CC} 
Suppose $\H$ is subcritical, i.e., $k<n$. 
\begin{enumerate}
	\item[(i)] There are $n-k$ simple Reeb orbits of $\a_{\H}$, all of which are good. 
	\item[(ii)] Given $m>0$, a deformation of $\H$ can be made such that for $*<m$, 
	 $$\rk(\C_*(\H,\a_{\H}))= \begin{cases} 1, &\mbox{ if } *=2n-k-4+2i, \text{ for some } i\in \mathbb{N}, \\
																				 0, &\mbox{ otherwise.}																					 
\end{cases}$$ 
\end{enumerate}
\end{theorem}

\begin{remark}
In Theorem \ref{thm:Yau CC}, the shape of $\H$ can be chosen for each  $m>0$, such that the contact chain groups $\C_{*}(\H,\a_{\H})$, for $*<m$, are essentially generated by just one of the simple periodic Reeb orbits and its positive iterations. For each $m$, this Reeb orbit is  called the \textit{principal (Reeb) orbit} of the contact handle and denoted here by $x(m)$. The degree of $x(m)^i$ is $2n-k-4+2i$.
\end{remark}

The  Conley-Zehnder index of a closed Reeb orbit is  taken with respect to  the standard  trivialization $\Phi_{\H}$ of the stabilized contact bundle $\comp R_{\H}\oplus \xi_{\H}$. Using $\Phi_{\H}$, the (generalized) Conley-Zehnder index is defined for any Reeb trajectory not necessarily closed. The next lemma shows that not only is  $\H$  index-positive, but the positive constant can also be chosen by \textit{thinning} $\H$. This means that a tubular neighborhood of the belt sphere is thinned by choosing smaller coefficients $c_l$ of the handle's defining function $h$; see Definition \eqref{eq:handle} of $h$.

\begin{lemma}[Lemma 3.3 of  \cite{Yau}] 
\label{lemma:Yau-IP}
Given any positive number $N$, we can \textit{thin} $\H$ enough such that 
$$\MUCZ(\gamma;\Phi_{\H})> N\cdot T - 2n,$$
for any Reeb trajectory of $\a_{\H}$  with action $T$.
\end{lemma}
Note that it is crucial that the  handle is subcritical for this lemma to hold.

\begin{remark} We have a version of Lemma \ref{lemma:Yau-IP} for the unitary index: Let $\s_{\H}$ be a section  of $S^1[(\Lnc\xi_{\H})^{\otimes 2}]$ obtained from $\Phi_{\H}$. Then  given any $N>0$, we can \textit{thin} $\H$ enough such that 
\begin{equation}
\label{eq:WIPhandle}
\mu(\gamma;\s_{\H})>N\cdot T-a,
\end{equation}
for any $R_{\H}$-trajectory with action $T$. 

This follows from Remark \ref{rmk:IP}. Note that  since $\pi_1(\H)=0$, the handle is weakly index-positive with respect to any section of $S^1[(\Lnc\xi_{\H})^{\otimes 2}]$. 
\end{remark}

\subsection{Weak index-positivity and subcritical contact surgery}
\label{sec:WIPHA}

\begin{definition} 
Suppose $(M',\xi')$ is a contact manifold obtained by glueing a contact manifold $\H\cong D^k \times S^{2n-k-1}$ to $(M,\xi)$. If  $c_1(\xi')=0$ and if the section $\s$ of $S^1[(\Lnc\xi)^{\otimes 2}]$ extends over $\H$, then we say that the surgery is \textit{compatible} with $\s$.
\end{definition}

Recall from  Remark \ref{remark:s'}, a given section $\s$ is almost always compatible with a given surgery.

\begin{theorem}
\label{thm:WIP} 

			   Let $(M,\xi)$ be a weakly index-positive with respect to a section $\s_M$ of $S^1[(\Lnc\xi)^{\otimes 2}]$. 
			   Suppose $(M',\xi')$ is a contact manifold obtained from $M$ by a subcritical $k$-surgery 
	       compatible with $\s_M$. Then for every integer $r$, there exists a non-degenerate contact form $\a'_r$ such that  the following hold:

\begin{enumerate}
	\item  $(M',\a_r')$ is weakly index-positive with respect to some extension $\s'$ of $\s_M$;
	\item  if $c_j$ and $c'_j$ denote the number of degree $j$ generators of $(M,\a)$ and $(M',\a'_r)$, then for $j \leq r$,
				 				
				 $$c'_j = c_j + b_j,$$  where 

				 $$b_j= \begin{cases} 
								1, & \mbox{if } j=2n-k-4+2m, \mbox{ for }m\in \mathbb{N}, \\
								0, & \mbox{otherwise}.
								\end{cases} $$

\end{enumerate}
\end{theorem}

\begin{proof}
We first prove that weak index-positivity is preserved under compatible surgery. Let $\a$ be a weakly index-positive contact form for $\xi$, with respect to a section $\s_M$.  By definition, there exists constants $\k>0$ and $\kk$ such that $\mu(\gamma_{M})\geq \k \cdot T+\kk$, for any Reeb trajectory $\gamma_M$ with action $T$.

Let $\gamma$ be a new Reeb trajectory created when $\H$ of index $k<n$ is glued to $M$.  Suppose  $p$ segments of $\gamma$ lie in $M$. Then the number of segments of $\gamma$ in $\H$ is between $p-1$ and $p+1$, and without loss of generality, assume this number is $p+1$. Denote the pieces in $M$ and $\H$ by $\{\gamma_i\},\ i=1, \ldots, p$, and $\{\gamma_{j}\},\ j=p+1, \ldots, 2p+1$, respectively.   

Let $S^{k-1}_M$ denote the  isotropic attaching sphere in $M$. Since the submanifold $ S^{k-1}_M$ has codimension strictly greater than $(\dim M-1)/2$, we have the following: Given any $K>0$, we can choose a neighborhood $\mathcal{U}_M$ of $S^{k-1}_M$  such that any Reeb trajectory passing through $\mathcal{U}_M$ has action greater than $K$; for instance, see\cite{Yau}. 
Then for any $K>0$, we thin $\H$ to fit into $\mathcal{U}_M$  and we get 
$$\A(\gamma)=\sum_{i=1}^{p}\A(\gamma_i)+\sum_{j=p+1}^{2p+1}\A(\gamma_j)\geq Kp, \ \text{ and so } \ p\leq \frac{\A(\gamma)}{K}.$$ 

Suppose $\s'$ is an extension of $\s_M$ to $\H$, that is, $\s'|_{M}=\s_{M}$ and  $\s'|_{\H}$ is a  section of $S^1[(\Lnc\H)^{\otimes 2}]$.    Recall that $\s_{\H}$ and $\s'|_{\H}$ are homotopic since $\pi_1(\H)=0$.  To estimate $\mu(\gamma)$, we consider the sum of the local indices of all its pieces $\sum\mu(\gamma_i;\s_M) + \sum\mu(\gamma_j;\s_{\H})$. We must take into account the two types of discrepencies between  local indices and the actual index of a new Reeb trajectory. One type of discrepency  arises  from glueing at each of the $2p$ points of attachment. By the $\mu-$catenation lemma,  these error terms are bounded by  some constant $b$ independent of the trajectories.  The second type of discrepency occurs if   $s'|_{\H}\neq s_{\H}$. However, Remark \ref{remark:s} gives us $ |\mu(\gamma_j;\s'|_{\H})-\mu(\gamma_j;\s_{\H})|\leq c$, for some constant $c$  independent of the trajectory. Let $N$ and $a$ be the weak index-positivity contants of the handle from \eqref{eq:WIPhandle}, and choose $N>\kappa_1$. Then

\begin{align*}
\mu(\gamma;\s') & \geq \sum_{i=1}^{p}{\big(\k\A(\gamma_i)+\kk\big)}+\sum_{j=p+1}^{2p+1}{\big(N\A(\gamma_j)-a\big)}-b(2p)-c(p+1)\\
      			& \geq \k \bigg(\sum_{i=1}^{p}\A(\gamma_i)+\sum_{j=p+1}^{2p+1}\A(\gamma_j)\bigg) +\kk p - a(p+1) - 2bp -c(p+1)\\
      			& \geq \k \A(\gamma)-(|\kk|+a+2b+c)p-a-c\\
      			& \geq \k \A(\gamma)-(|\kk|+a+2b+c)\bigg(\frac{A(\gamma)}{K}\bigg)-a-c\\  
    		  	& =\left(\k-\frac{|\kk|a+2b+c}{K}\right)\A(\gamma)-a-c.
\end{align*}

Take $K>(|\kk|a+2b+c)/\k$.  Set $\k'=\k-(|\kk|a+2b+c)/K$ and $\kk'=-a-c$.  Both $\k'$ and $\kk'$ are independent of the trajectory therefore the new contact manifold $(M',\a')$  is weakly index-positive with respect to $\s'$.  Since weak index-positivity is stable under small perturbations (Lemma \ref{lemma:WIP-pert}), we may assume $\a'$ is non-degenerate.

To prove ($2$), for a fixed $r\in \mathbb{N}$, we find a contact form $\a'_r$ for $\xi'$ such that a closed Reeb trajectory created from the handle attachment has degree greater than~$r$.  Given $m>0$, a contact form  $\a_{\H}$ for $\H$ can be chosen such that $\C_*(\H)$ for $*\leq m$, is generated by the principal Reeb orbit $x(m)$ along with its iterates. Let $C$ be the constant from $\ie$ such that $|\MUCZ(x)-\mu(x;\s')|\leq C$, for a closed Reeb orbit $x\in M'$. If we take 
$m = r-C$, we can show that any periodic Reeb orbit that does not lie entirely in $M$ or $\H$ has unitary index greater than $r-C$ and hence  Conley-Zehnder index greater than $r$.  Let $\gamma$ be such a trajectory and suppose it is made of $p$ pieces of non-periodic trajectories in $M$ and $p$ pieces inside $\H$. Denote the segments in $M$ and $\H$ by $\gamma_1, \ldots, \gamma_p$ and $\gamma_{p+1}, \ldots, \gamma_{2p}$, respectively.  Let $a,b,c,N,\k$ and $\kk$ be the constants  from the proof of part (1). We get

\begin{align*}
\mu(\gamma;\s') & \geq \sum_{i=1}^{p}{\big(\k \A(\gamma_i)+\kk \big)}+\sum_{j=p+1}^{2p}{\big(N\A(\gamma_j)-a\big)}-b(2p)-c(p)\\
								& \geq \k\cdot \sum_{i=1}^{p} \A(\gamma_i) + N \cdot \sum_{j=p+1}^{2p}\A(\gamma_j)+\kappa_{2}p - a2p - 2bp -cp.\\
  					    %& \geq \k K p + N \cdot \sum_{j=p+1}^{2p}\A(\gamma_j) -(|\kk|+2a+2b+c)p\\
      					%& \geq \big(\k K  -(|\kk|+2a+2b+c)\big)p\\
      					%& \geq \big(\k K  -(|\kk|+2a+2b+c)\big).\\
\end{align*}

\noindent Recall for $K>0$, we can thin $\H$ to ensure that $\A(\gamma_i)>K$. Then we have

\begin{align*}
\mu(\gamma;\s') 
  					    & \geq \k K p + N \cdot \sum_{j=p+1}^{2p}\A(\gamma_j) -(|\kk|+2a+2b+c)p\\
      					& \geq \big(\k K  -(|\kk|+2a+2b+c)\big)p\\
      					& \geq \big(\k K  -(|\kk|+2a+2b+c)\big).\\
\end{align*}
We can choose $K$ that satisfies 
$$K>\dfrac{r-C+|\kappa_2|+2a+2b+c}{\kappa_1},$$  
further adjusting $\H$ if necessary.  Denote this particularly shaped $\H$ by $\H_r$. Then by attaching  $(\H_r,\a_{\H_r})$, to $(M,\a_r)$,  we get $\mu(\gamma;\s')>r-C$ for any new periodic trajectory $\gamma$ of $M'$. To finish the proof,  note that the principal Reeb orbit of $\H_r$ and its iterations $x(r)^m$ have degree $2n-k-4+2m$.

\end{proof}

For each $r>0$, let $(\H_r,\a_{\H_r})$ be the deformation of $\H$ as in the proof of Theorem \ref{thm:AFSHA}. The mean index of the principal Reeb orbit of $\H_r$ is $\Delta(x(r))=2$, for each $r$, therefore $\Delta_{x}=\lim_{r\to +\infty} \Delta(x(r))=2$.  Also, the sign $\sigma(x(r))=(-1)^{|x(r)|}$ is independent of $r$. Thus we get the following.

\begin{corollary}
\label{cor:SHAMEC}
Let  $(M',\xi')$ be a contact manifold obtained by a $k$-surgery on $(M,\xi)$ as in Theorem \emph{\ref{thm:WIP}} and assume the dimension of $M$ is greater than $3$. Whenever $\chi(M,\xi)$ is defined, we have
$$\chi(M',\xi')=\chi(M,\xi)+(-1)^{k}\frac{1}{2}.$$ 
\end{corollary}

\begin{remark}
We remind the reader that Corollaries \ref{cor:AFSHAMEC} and \ref{cor:SHAMEC} hold in the dimension $3$ when the MEC is defined via linearized contact homology. See Section \ref{sec:lch}. For more about the issue in the cylindrical version, see Section \ref{sec:dim 3}.
\end{remark}

\subsection{Proof of Theorem \ref{thm:AFSHA}}
\label{sec:pf-AFSHA}
We prove that asymptotic finiteness and weak index-positivity together are preserved under subcritical contact surgery.

Let $(M,\xi)=\{((M,\a_r))\}$ be an asymptotically finite contact manifold that satisfies the assumptions of Theorem \ref{thm:AFSHA}. For each $r$, the contact form $\a_r$ has principal Reeb orbits  $\{\gamma_1(r),\cdots,\gamma_l(r)\}$ which, along with their iterations, generate $\C_*(M,\a_r)$ up to degree $d_r$.  Given any integer $d'_r$, it is shown in  the proof of Theorem \ref{thm:WIP} that there exists a weakly index-positive contact form $\a'_{r}$ for $\xi'$ such that for any closed Reeb orbit $\gamma$ with $|\gamma|\leq d'_r$, the orbit is  an iteration of either the principal Reeb orbit $x(r)$ of $\H$ or $\gamma_i(r)$, for some $1 \leq i \leq l$.  Take $d'_r=d_r$, then $\{(M',\xi')=\{(M',\a'_r)\}$ is a sequence that does the job.

\begin{remark}
Corollary \ref{cor:AFSHAMEC} immediately follows from Theorem \ref{thm:AFMEC} and Corollary~\ref{cor:SHAMEC}.
\end{remark}

\subsection{Cylindrical MEC and contact surgery in dimension $3$}
\label{sec:dim 3}
We now elaborate a little bit more on the exclusion of dimension $3$ for the cylindrical version of  Corollaries \ref{cor:AFSHAMEC} and \ref{cor:SHAMEC}.   Cylindrical contact homology theory generally  requires that there are no contractible periodic Reeb orbits of index $-1,0,$ and $1$.  However, taking the connect sum of two contact $3$-manifolds introduces one (simple) contractible Reeb orbit which is of degree $1$. This is precisely the principle Reeb orbit of the \textit{connecting tube} $\H$ and it lies on the coisotropic $2$-sphere, see Section \ref{sec:surgery}.  The principle orbit is the generator of $\C_{*}(\H,\a_{\H})$ whose degree has been calculated in \cite{Yau} to derive Theorem \ref{thm:Yau CC}.  So we may verify here using Theorem \ref{thm:Yau CC} that this Reeb orbit indeed has index $1$ simply by setting $k=1, n=2$, and $i=1$. 

Even though the new manifold has a contractible Reeb orbit of degree $1$, this fact does not necessarily prevent its cylindrical contact homology from being defined.  (See for instance \cite{AM} where it is shown that certain contact forms which admit such orbits can still be considered.) Currently, we are working to prove that Corollaries \ref{cor:AFSHAMEC} and \ref{cor:SHAMEC} also hold in dimension $3$.

\subsection{Contact chain groups and subcritical contact surgery}
\label{sec:chha}

Since the conditions of  Theorem \ref{thm:WIP} do not require the chain complex $\C_j(M,\a)$ to be generated by a finite number of Reeb orbits,  the proof of Theorem \ref{thm:AFSHA} holds	for a more general class of contact manifolds. Suppose there is a sequence of non-degenerate contact forms $\{\alpha _r\}$ for $\xi$ such that the contact chain complex of each $\a_r$ is essentially generated by a set of simple Reeb orbits $\{\gamma_1(r),\ldots, \gamma_m(r), \ldots\}$ together with their iterations as $r\rightarrow \infty$. 

\vs
More precisely, suppose this sequence $\{\alpha _r\}$ satisfies the following conditions.
\begin{itemize}

	\item [(1)] For each $\alpha_r$, there is a set of simple Reeb orbits $\{\gamma_1(r), 		
				 			\ldots, \gamma_m(r)\}$, where the number of orbits $m$ increases  as $r$ increases. 
	
	\item [(2)] There is an increasing sequence of integers $\{d_r\}$ such that if $\gamma \in \mathcal{P}^{d_r}(\alpha_r)$, then 		
							$\gamma$ is an iteration of some $\gamma_i(r)$,
							i.e., $\gamma=\gamma_i(r)^k$. 
	
	\item [(3)] The mean indices of each sequence converge:
							$\Delta(\gamma_i(r))\rightarrow \Delta_{i}$.
	
	\item [(4)] The sign $\sigma(\gamma_i(r))$ is independent of $r$.
\end{itemize}

\begin{definition}
The orbits  $\{\gamma_1(r),\ldots, \gamma_m(r)\}$  are called the  \textit{principal orbits} of $\a_r$, the limits $\Delta_{i}$ are called  \textit{asymptotic mean indices}, and we set $\sigma(\gamma_i(r))=\sigma_i$, c.f.~Definition \ref{def:AF}. 
\end{definition}

\begin{theorem}
\label{thm:WIPAG} 			   
				 Let $(M,\xi)=\{(M,\a_r)\}$ be a closed contact manifold that satisfies the conditions (1)-(4) and suppose there is a section  $\s_M$ of 	
			   $S^1[(\Lnc\xi)^{\otimes 2}]$ such that  each  $a_r$ is weakly index-positive with 
				 respect to a section  $\s_M$. Let $(M',\xi')$ be a contact manifold obtained from $M$ by a subcritical surgery 
			   compatible with  $\s_M$. Then $(M',\xi')$ admits a sequence of contact forms $\{a_r\}$ that satisfies (1)-(4) and each  $a_r$ is weakly index-positive 	
			   with respect to an extension $\s'$. 
\end{theorem}

\section{Other variants of the MEC formula}
\label{sec:remarks}
In this section, we describe an extension of our results to closed Reeb orbits of any homotopy type.   We also discuss how our results translate to linearized contact homology. Throughout this section, all closed Reeb orbits are assumed to be non-degenerate.

\subsection{Homologically non-trivial Reeb orbits} 
\label{sec:non-contr}

\begin{remark} 
Suppose we want to consider the contact chain groups generated by homologically non-trivial Reeb orbits. If $\gamma$ is a such a Reeb orbit, then $\MUCZ(\gamma)$ depends on the trivialization of $\xi$ along $\gamma$ and one must keep track of such choices. We first mention the convenience of using a fixed section $\s$ of $S^1[(\Lnc \xi)^{\otimes 2} ]$ to consistently assign such trivializations  along closed paths in $M$ of any homotopy class. Suppose $\gamma$  is $T$-periodic.  A unitary frame $F=\{e_1, \cdots, e_{n-1}\}$ of $\xi|_{\gamma}$ can be chosen such that its associated section $\s_F=(\Lambda^{n-1}_{j=1}e_j )^{\otimes 2}$ is homotopic to $\s|_{\gamma}$. The Conley-Zehnder index of $\gamma$ can be defined with respect to the homotopy class $[\s]$. By this, we mean we can choose any frame $F$ such that $[\s_F]=[\s_{\gamma}]$.   Since the Conley-Zehnder index of $\gamma$ depends only on the homotopy class of the frame, we need to show that all such frames are  homotopic.   First, fix a frame $F$ such that $\s_{F}=\s|_{\gamma}$. Note that all such frames are homotopic.  Next, let $\s'$ be a section homotopic to $\s$  and let $F'$ be a unitary frame with $\s_{F'}=\s'|_{\gamma}$.  We can write $\s'(x)=e^{if(x)}\s(x)$ for some continuous function $f\col M\rightarrow \real$.  Since $f(\gamma(t))=f(\gamma(t+T))$ for all $t,$ the frames $F'_{\gamma}$ and $F_{\gamma}$ are homotopic. 
  
Note that using $\s$ to choose trivializations along closed Reeb orbits  eliminates the ambiguity in the definition of contact homology. Usually, one takes a set of representatives for a basis of $\HH_1(M;\mathbb{Z})$ and chooses a trivialization of $\xi$ along each representative . Then this set of trivializations is extended to $\xi$  along  closed Reeb orbits; see \cite{Bo}. We can use $\s$ to consistently choose trivializations of $\xi$ for each class in $\HH_1(M;\mathbb{Z})$. This choice of framings coincides with directly assigning a frame to $\xi$ along a closed Reeb orbit with $\s$. 
\end{remark}

\begin{remark} 
Note that any frame along $\gamma$ chosen by $\s$ is closed under iteration. From this, we see that the mean Euler characteristic formulas   hold when including closed Reeb of any homotopy type. In this case, we just specify the choice of $\s$ and take the mean indices of the orbits with respect to $\s$.
\end{remark}

\subsection{Linearized contact homology} 
\label{sec:lch}

The mean Euler characteristic can be set via  linearized contact homology. Details of cylindrical, full, and linearized versions are included in a survey of contact homology in \cite{Bo:survey}.   

Let  $\textbf{A}_*(M,\a)$ denote the differential graded algebra generated by the good $R_{\a}$-Reeb orbits, and denote by $\partial:=\sum\partial^{j}$, the corresponding (full) differential. The full contact homology  $\HC_*(M,\xi)$ is the homology of $(\textbf{A}_*(M,\a),\partial)$ and is an invariant of the contact structure $\xi$.  A linearized contact homology  $\HC^{\ep}_*(M,\a)$ exists when  $\textbf{A}_*(M,\a)$ admits an \textit{augmentation} $\ep$; this is an algebra homomorphism that satisfies $\ep(1)=1$ and $\ep \circ \partial=0$. The \textit{linearized differential} $\partial^{\ep}$ is defined with respect to $\ep$ and the contact chain groups $\C_{*}(M,\a)$ are the same as in the cylindrical version,  i.e., these groups are  generated by the good Reeb orbits and graded by $|\cdot|=\MUCZ(\cdot)+n-3$. Unlike cylindrical contact homology, the linearized version is defined with no assumptions on contractible Reeb orbits of degree $-1,0$ and $1$.  One should note however, $\HC^{\ep}_{*}(M,\a)$ is generally \textit{not} a contact invariant of $(M,\xi)$ as it depends on the augmentation.

Linearized and cylindrical contact homology are related in the following case: Suppose $\C^{0}_1(M,\a)=0$, i.e., there are no  contractible closed Reeb orbits of degree $1$.  Then $\textbf{A}_*(M,\a)$ admits a \textit{trivial} augmentation $\ep_0$; this is the algebra homomorphism defined by $\ep_{0}(1)=1$ and $\ep_0 (\gamma)=0$.  In addition, if $\C^{0}_*=0$, for $*=-1,0$, then $\HC^{\ep}_*(M,\a)$  coincides with $\HC^{cyl}_*(M,\xi)$.

A class of examples that come with a more geometrical form of this extra structure are symplectically fillable contact manifolds; see  \cite{Bo:survey, BO}. Suppose $(W,\omega)$ is a  symplectic filling of $(M,\a)$. Then $\textbf{A}_*(M,\a)$ admits a natural augmentation denoted here by $\ep^{W}$. This linearized contact homology  $\HC^{W}_{*}(M,\a)$ depends only on the filling $(W,\omega)$ of $(M,\xi)$, and so it is often denoted by $\HC_*(W,\omega)$. 

\begin{remark} 
The definition of an asymptotically finite manifold can  be naturally modified to suit symplectically fillable contact structures.  Also, a filling and its augmentation can be carried across a surgery. Since the proofs of Theorems \ref{thm:AFMEC} and \ref{thm:AFSHA} do not explicitly make use of the differential,  both MEC formulas \eqref{eq:AFMEC} and \eqref{eq:AFSHAMEC}  again hold. We also note that in this case,  the MEC is indeed a contact invariant of $(M,\xi)$.
%\end{remark}

%\begin{remark}
Moreover, the expression of the change of the MEC under connect sums also holds in dimension $3$ for linearized contact homology.
The addition of the contractible Reeb orbit of degree~$1$, the principle Reeb orbit of the connecting tube between two contact $3$-manifolds, does not pose a problem for this version of the mean Euler characteristic. 
\end{remark}

\section{Contact manifolds of Morse-Bott type}
\label{sec:MB}

The goal of this section is to prove  Theorem \ref{thm:MBMEC} and obtain a formula for the mean Euler characteristic in the Morse-Bott case. We set some notation and include preliminaries of the Morse-Bott  approach to contact homology. We set notation and terminology for Reeb orbifolds and we also define the mean index of a Reeb orbifold   and discuss the orbifold invariant $\eul$.

\subsection{Morse-Bott contact forms and Reeb orbifolds} 
\label{sec:MB-prelims}

Recall that contact homology requires a non-degenerate contact form $\a$. To consider  natural and symmetric forms, we use the Morse-Bott approach  developed in \cite{BoPhD} to work with these degenerate cases.  This approach to contact homology is a Morse-Bott theory for the action functional  $\A(\gamma)=\int_{\gamma}\a$ on the loop space of $M$, in other words,  we can consider contact forms such that the critical points of $\A$ form smooth submanifolds. 

\begin{definition} 
A contact form $\alpha$ for $(M,\xi)$ is said to be of \textit{Morse-Bott type} if the action spectrum $\sigma(\alpha)$ of $\alpha$ is discrete and if, for every $T\in\sigma(\alpha)$, {\nolinebreak $N_T=\left\{p\in M\mid\varphi_T(p)=p\right\}$} is a closed, smooth submanifold of $M$, such that rank $d\alpha\mid_{N_T}$ is locally constant and $T_pN_T=ker(\varphi_{T}-I)_p$.
\end{definition}

Denote the quotient of $N_T$ under the circle action induced by the Reeb flow by $S_T$; this is an orbifold with singularity groups $\mathbb{Z}_k$ and the singularities correspond to Reeb orbits of period $T/k$, covered $k$ times. We call these orbit spaces \textit{Reeb orbifolds}.

\begin{flushleft}

We  set some orbifold terminology used in this paper. For simplicity, assume that each $S_T$ is connected.
\begin{itemize}	
	\item Let $kS_T$ denote the $k$-fold cover of $S_T$.
				(We use this notation to distinguish the cover from the orbit space of $kT$-periodic Reeb orbits $S_{kT}$.)
	\item An orbifold $S_T$ is called \textit{simple} if there is a point $x\in S_T$ with minimal period $T$.
	\item If every orbit of $S_T$ has minimal period $T$, then $S_T$ is called a \textit{minimal} orbifold.
	\item Suppose $S_T$ is a simple orbifold. If $kS_T\subseteq S_{kT}$ implies  $kS_T=S_{kT}$, we call $S_T$ a \textit{maximal} orbifold. 
				Equivalently, no multiple cover $kS_T$ is contained in some larger simple orbit space. 
\end{itemize}

In this section, we use the  notation below to distinguish between the following indices for closed Reeb orbits $\gamma$:
\begin{itemize}
	\item $\MUCZ(\gamma)=$ the Conley-Zehnder index if $\gamma$ is non-degenerate, 
	\item $\MURS(\gamma)=$ the generalized Conley-Zehnder index if $\gamma$ is degenerate.
\end{itemize}

\end{flushleft}

\vs
Given any two Reeb orbits $\gamma_1$ and $\gamma_2$ of the same  orbit space $S_T$, we have $\MURS(\gamma_1)=\MURS(\gamma_2)$.  Therefore, the index of the orbifold defined as $\MURS(S_T):=\MURS(\gamma)$, for any $\gamma \in S_T$.  This number is a half-integer, but $\MURS(S_T)-\frac{1}{2}\dim(S_T)$ is always an integer.  The degree of a Reeb orbit $\gamma \in S_T$  is set as $$|\gamma|^{\scriptscriptstyle{RS}}:=\MURS(S_T)-\frac{1}{2}\dim S_T +n-3.$$ This last orbit index  depends only on $S_T$, thus  we  set $|S_T|:=|\gamma|^{\scriptscriptstyle{RS}}$, for any $\gamma\in S_T$, and we also attach a sign to each Reeb orbifold by  $$\sigma(S_T):=(-1)^{|S_T|}.$$

\begin{definition}
A closed Reeb orbit $\gamma$ is said to be \textit{bad} if it is a  $2m$-cover of a simple Reeb orbit $\gamma'\in S_T$ and if $(\MURS(S_{2T})\pm \frac{1}{2} \dim S_{2T})-(\MURS(S_{T})\pm \frac{1}{2} \dim S_{T})$ is odd. If a Reeb orbit $\gamma$ is not bad, we say it is \textit{good}. 
\end{definition}

We also say that a Reeb orbifold is \textit{bad} if it contains a bad Reeb orbit. A Reeb orbifold is otherwise called \textit{good}. 
These definitions above are extended from the definitions of \textit{good} and \textit{bad} Reeb orbits in  Section \ref{sec:preliminaries}.

\subsection{The Morse-Bott chain complex}
\label{sec:MBCC}

Using a Morse function, a Morse-Bott type contact form can be perturbed in such a way that all Reeb orbits are non-degenerate and correspond to the functions critical points.  The  critical points are the orbits that generate the Morse-Bott chain complex. 

This Morse function is constructed inductively in \cite{BoPhD} as follows. For the smallest $T\in\sigma(\alpha)$, the orbifold $S_T$ is a manifold. Let $f_T$ be any Morse function on $S_T$.  For larger $T\in\sigma(\alpha)$, the orbifold $S_T$ has singularities  $S_{T_i}$ with $T_i|T$.  Extend the previously defined Morse function $f_{T_i}$ on $S_{T_i}$ to a function $f_T$ on $S_T$ such that the Hessian of $f_T$ restricted to the normal bundle of $S_{T_i}$ is positive definite. These Morse functions are then lifted to $N_T$ and extended to a function $\bar{f}_T$ on M such that they have support only on a tubular neighborhood of $N_T$.

Consider the perturbed contact form $\a _{\lambda} = (1+\lambda \bar{f}_T)\a,$ for $T\in\sigma(\a)$ and a small postive constant $\lambda$.

\begin{lemma} \cite{Bo}. 
For all $T_0$, we can choose $\lambda>0$ small enough such that the periodic orbits of $R_{\alpha_{\lambda}}$ in $M$ of action $T\leq T_0$ are non-degenerate and correspond to the critical points of $f_{T}$.
\end{lemma}

Let $x\in S_T$ be a critical point of $f_T,$ denote by $\gamma_x$ its corresponding non-degenerate $R_{\a_\lambda}$-Reeb orbit, and let $(\gamma_x)^k$ be its $k^{th}$ iteration for ($k=1,2,\ldots$). For small $\lambda,$ the Conley-Zehnder index of $(\gamma_x)^k$ can be computed by
\begin{equation}
\label{eq:MBMUCZ}
\MUCZ((\gamma_x)^k)=\MURS(S_{kT})-\frac{1}{2}\dim S_{kT} + ind_x(f_{kT}),
\end{equation}
where $ind_x(f_{kT})$ denotes the Morse index. 
Note that since $kS_T \subset S_{kT}$, the index $ind_x(f_T)$ is constant under iterations of $x$. The degree of a critical point of $f_T$ is set as $$|x|:=\MUCZ(\gamma_x)+n-3.$$

The Morse-Bott chain complex $\C^{MB}_{*}(M,\a)$ is generated by critical points $x\in S_T$ of the Morse functions $f_T$ and graded by $|x|$.  For the definition of the Morse-Bott differential $\partial^{MB}$; see \cite{Bo}.

For Reeb orbifolds, we can also introduce the notion of the mean index. Note by  \eqref{eq:MBMUCZ} and $\ia$, we have
\begin{equation}
\label{eq:MBMI}
\Delta(x)=\lim_{k \to \infty} \frac{\MURS(S_{kT})}{k},
\end{equation} for any $ x\in S_T$.
The right hand side of \eqref{eq:MBMI} is independent of  the orbit $x\in S_T$, and so we define the \textbf{mean index of a Reeb orbifold $S_T$} as 
\begin{equation}
\label{eq:Orb-MI}
\Delta(S_T):=\lim_{k \to \infty} \frac{\MURS(S_{kT})}{k}\ . 
\end{equation} 

Let us now turn to the relation between the homology $H_*(\C^{MB}_*(M,\a),\partial^{MB})$ and cylindrical contact homology. To compute the mean Euler characteristic set via  cylindrical homology, we restrict our attention to Morse-Bott type contact forms that satisfy the conditions of following theorem.

\begin{theorem}
\label{thm:MB}\cite{Bo} 
Let $(M,\alpha)$ be a contact manifold of Morse-Bott type for a contact structure $\xi$ on $M$ satisfying $c_1(\xi)=0$. Suppose that, for $T\in\sigma(\alpha)$, $N_T$ and $S_T$ are orientable, $\pi_1 (S_T)$ has no disorienting loops, and all Reeb orbits are good. Assume that the almost complex structure $J$ is invariant under the Reeb flow on all submanifolds $N_T$. Assume that the cylindrical homology is well-defined: $\C^{0}_* =0$ for $*=-1,0,1.$ Assume that there exist $c>0, c'$ such that $|\MURS(S_T)|>cT+c'$  for all orbit spaces $S_T$ of contractible periodic orbits, and that there exists $\Gamma_0 <\infty$ such that, for every Reeb trajectory leaving a small tubular neighborhood $U_T$ of $N_T$ at $p$, we have $\varphi_t(p)\in U_T$ for some $0<t<\Gamma_0$.  Then the homology $H_*(\C^{\bar{a}}_{*},\partial^{MB})$ of the Morse-Bott chain complex is isomorphic to the cylindrical contact homology $\HC^{cyl}_*(M,\alpha)$.
\end{theorem}

\begin{remark}
\label{remark:IPN-orb}
The existence assumptions of such $c>0, c'$ and  $\Gamma_0$ in Theorem \ref{thm:MB} are known as \textit{index-positivity/negativity} and \textit{bounded return time} conditions. These two assumptions guarantee that no extra orbits of the perturbed contact form $\alpha _{\lambda}$ contribute to the homology and that  $\partial^{MB}$ correponds to $\partial^{cyl}$. 

We note however, that $\partial^{MB}$ does not enter  the proof of Theorem \ref{thm:MBMEC}, and so the Morse-Bott MEC formula holds whenever the homology  $H_*(\C^{\bar{a}}_{*},\partial^{MB})$ is indeed defined. For example,  if the index-positivity/negativity condition is dropped, we still have a homology of the Morse-Bott chain complex, although it may no longer be isomorphic to $\HC^{cyl}_*(M,\alpha)$, cf. Remark \ref{remark:index-positivity}. For this reason, we include the possibility that there are Reeb orbifolds with zero mean index in the proof of Theorem \ref{thm:MBMEC}, even though it is not necessary due to the index-positivity/negativity assumption.
\end{remark}

\begin{remark}
To apply Theorem \ref{thm:MB}, a contact manifold of Morse-Bott type can only have good Reeb orbits.  This assumption cannot be dropped as it is required for the homology to be defined. These manifolds have no bad Reeb orbifolds, and therefore,  the Morse-Bott version of the  MEC formula has a slightly more simple form than the other versions. 
\end{remark}

\begin{remark}
The class of Morse-Bott type contact manifolds does not behave under handle attaching. Breaking a closed Reeb orbit during surgery cannot be avoided for  some contact manifolds; some examples are  the standard contact sphere and the Ustilosky spheres; see Section \ref{sec:examples}. However,  a contact manifold  in this class with finitely many simple Reeb orbifolds   is asymptotically finite. Therefore, we can combine the Morse-Bott MEC formula \eqref{eq:MBMEC} with Corollary \ref{cor:AFSHAMEC} to compute the mean Euler characteristic of a manifold obtained by a subcritical contact surgery, whenever the conditions of Theorems \ref{thm:MBMEC} and \ref{thm:WIP} are met. 
\end{remark}

\begin{remark}
In \cite[Question 4.10]{VK2}, there is an observation that for a Brieskorn manifold $X_a=\Sigma(a_0,\ldots,a_n)/S^1$, see also Section \ref{sec:examples}, contact homology can be related to Chen-Ruan's orbifold cohomology. It is suggested there that this might give some insight on how to compute contact homology for more general $S^1$-bundles over symplectic orbifolds.
\end{remark}

\subsection{Morse theory on the orbit spaces}
\label{sec:orb-Morse-theory}
This section  follows closely to \cite[Section 2.3]{Bo}.

For an orbifold $S_T$, the strata are the connected components of the sets with the same isotropy. Let $\{S_{T'}\}$ be the strata of simple orbifolds of $S_T$ (i.e. $T'|T$ and $S_{T'}\subseteq S_T$), and let $f_{_T}$ be the Morse functions on the orbit spaces $S_T$ defined in the previous section. 

\begin{definition}
A Morse function $f_T$ on the orbit space is called \textit{admissible} if, for every critical point $\gamma$ of $f_T$ with minimal period $T/k$, the unstable manifold of $\gamma$ is contained in $S_{T/k}$.
\end{definition}

The Morse functions $f_T$ are admissible by construction  and so Morse theory can be extended to orbifolds as seen by this next proposition.  

\begin{proposition}
[\cite{Bo}] \label{prop:orb-MT} 
If the Morse function $f_T$ is admissible on the orbit space $S_T$, then the Morse complex of $f_T$ is well-defined and its homology is isomorphic to the singular homology of $S_T$.
\end{proposition}

\subsection{The orbifold invariant $\eul$}
\label{sec:orb-inv}

Let $G$ be an arbitrary compact Lie group and $M$ be a smooth $G$-manifold.  It is shown in \cite{Ill} that $M$ can be given an equivariant CW structure.  Suppose $M$ is also compact and that $G$ acts smoothly, effectively, and almost freely. Then $M$ is triangularizable as a finite $G$ CW-complex and $\X=M/G$ is a compact quotient orbifold with finitely many cells.

\begin{flushleft}
Set $\eul(\X)$ to be the following invariant:
\begin{equation}
\label{eq:def e}
	\eul(\X):=\underset{\bar{\sigma}}{\sum}(-1)^{\dim \bar{\sigma}}|Stab(\bar{\sigma})|,
\end{equation}
where the sum runs over the $q$-cells $\bar{\sigma}$ of $\X$ and $|Stab(\bar{\sigma})|$ is the order of the stabilizer subgroup of $\bar{\sigma}$ in $G$.
\end{flushleft}

We can calculate $\eul(\X)$ in the following way.
Let $\chi(\X)$ denote the Euler characteristic of the underlying space.  
First, for all minimal orbifolds $\X$, set $\CHI(\X)=\chi(\X)$. Next, consider any orbifold $\X$, the strata of $\X$ are the connected components of sets with the same isotropy subgroup.  Let $\{\X_i\}$ be the strata of simple orbifolds of $\X$ then we define $\CHI(\X_i)$ recursively as
\begin{equation} \label{CHI}
\CHI(\X_i):=\chi(\X_i)-\underset{\X_i'}{\sum}\CHI(\X_i'),
\end {equation}
where the sum runs over the  strata $\{\X_i'\}$ of simple orbifolds of of $\X_i$.

\begin{lemma} The orbifold invariant $\eul(\X)$ can be expressed as
\begin{equation}
\label{eq:euler}
\eul(\X)=\underset{\X_i}{\sum}\CHI(\X_i)\cdot|Stab(\X_i)|,
\end{equation}
where the sum runs over the strata $\{\X_i\}$ of $\X$.
\end{lemma}

\begin{proof}
\label{proof of lemma}
We start with the contribution to $\eul(\X)$ by each of its minimal strata.  Let $\X_0$ be a minimal orbifold and let $\{\bar{\sigma _0}\}$ be its q-cells. Then,
\[
\underset{\bar{\sigma _0}}{\sum}(-1)^{\dim \bar{\sigma _0}}|Stab(\bar{\sigma _0})|=\chi(\X_0)\cdot|Stab(\X_0)|.
\]
The additional contribution from any simple orbifold $\X_i\subset\X$ is
\begin{align}
\underset{simple\ \bar{\sigma _i}}{\sum} (-1)^{dim(\sigma _i)}|Stab(\bar{\sigma _i})| 
		& = \Big(\chi(\X_i) -\underset{\X_j\subseteq\X_i}{\sum}\CHI(\X_j)\Big)\cdot |Stab(\X_i)| \\
		& = \CHI(\X_i)\cdot |Stab(\X_i)|.
\end{align}
The left hand side sums over  simple cells. These are cells of the orbifold $\X_i$ that are not contained in any smaller orbifold  $\X_j\subset\X_i$.

Then summing over all the strata of $\X$, the lemma is proved.
\end{proof}

\begin{remark}
Let $\{S_{T_i}\}$ be  the simple strata of a Reeb orbifold $S_T$ and suppose $f_T$ is a Morse function on $S_T$ as in Section \ref{sec:MBCC}. By Proposition \ref{prop:orb-MT}, the Morse complex of $f_T$ coincides with the complex of the cell decomposition of the underlying space of $S_T$.  Then  $\eul(S_T)=\sum(-1)^{ind(x)}|Stab(x)|$, where the sum runs over all critical points $x$ of $f_T$, and for short, denote the Morse index by $ind(x)$. So now we have $\eul(S_T)=\sum\CHI(S_{T_i})\cdot|Stab(S_{T_i})|$.
\end{remark}

We refer the reader to \cite{ALR} for further details on this invariant $\eul$.

\subsection{Proof of Theorem \ref{thm:MBMEC}}
\label{sec:pf-MBMEC}

We prove the statement for $\chi^+(M,\xi)$ and the case for $\chi^-(M,\xi)$ is proved similarly.
Recall that the vector space $\C_l^{MB}(M,\alpha)$ is generated by the set of critical points $x$ of the Morse functions $\fT$ with degree $l=|x|=\MURS(S_T)-\frac{1}{2}\dim(S_T)+ind(x)+n-3$, for $T\in\sigma(\a)$. A contact form that satisfies the assumptions of Theorem \ref{thm:MB} has only good Reeb orbits and therefore all Reeb orbifolds are good. By $\ia$  and \eqref{eq:MBMUCZ}, for $x\in \C_l^{MB}\cap S_T$, we have
\begin{equation}\label{eq:1MB}
\big|\MURS(x^k)-\frac{1}{2}\dim(S_T)+ind(x)-k\Delta(x)\big|<n-1,
\end{equation}
and hence
\begin{equation}\label{eq:2MB}
-2<|x^k|-k\Delta(x)<2n-4.
\end{equation}

There are finitely many simple critical points since there are finitely many  simple orbit spaces by assumption. By this fact and the inequality (\ref{eq:2MB}), $\C_l^{MB}(M,\alpha)$ can be infinite dimensional only if $\Delta(x)=0$.  Although the  class of manifolds that have  Reeb orbits with $\Delta(x)=0$ do not meet all conditions of Theorem \ref{thm:MB} , we still cover this case  (see Remark \ref{remark:IPN-orb}). We mention that there exist integers $l_+$ and $l_-$ such that $\dim \HC_{\pm l}(M,\xi)<\infty$, for example, one can just take $l_+=2n-4$ and $l_-=-2$.

Considering the case $\Delta(S_T)>0$, denote the truncation of $\C_*^{MB}$ from below at $l_+$ and from above at $N$ for some $N>l_+$ by

\[
\C_l^{(N)}=
	        \begin{cases}
		      \C_l^{MB} &\text{if } l_+\leq l \leq N,\\
	 	      0   			&\text{otherwise}.
					\end{cases}
\]

\noindent The Euler characteristic of $\C_*^{(N)}$ is then
\[
\chi\big(\C_*^{(N)}\big)=\sum{(-1)^l \dim \C_l^{(N)}}=\sum^{N}_{l=l_+}{(-1)^l\dim \C_l^{MB}}.
\] 
\noindent The chain complex $\C_*^{(N)}$ is generated by the critical points $x\in \C^{MB}_*$ and its $k^{th}$ iterations, where $|x^k|\in [l_+,N]$.  The inequality (\ref {eq:2MB}) shows that $k$ ranges from some constant independent of $N$ up to roughly $N/\Delta(x)$. 

There are no bad Reeb orbits by assumption, and  so $\sigma(x^k)=\sigma(x)$, for all $k$. Then iterations of each $x$ contribute to $\chi\big(\C_*^{(N)}\big)$ by $N\sigma(x)/\Delta(x)+O(1)$ as $N\rightarrow\infty$.

Fix one maximal orbifold $S_T$, let $\{S_{T_i}\}$ be the set of simple strata of $S_T$, and denote the order of the stabilizer of $S_{T_i}$ by $p_i:=|Stab(S_{T_i})|$. Then iterations of any critical point $x'\in S_T$ of $\fT$ can be expressed as $x'=x^{p_i}$, where $x$ is a simple periodic orbit of $S_{T_i}$. Then the contribution to $\chi\big(\C^{(N)}_*\big)$ by iterations of each critical point $x$ of minimal period $T_{_i}$ can be expressed as
\[
\frac{N\sigma(x)}{\Delta(x)}=\frac{N\sigma(x)}{\Delta (x^{p_{_i}})}\cdot p_{_i}=
\frac{N\sigma (S_T)}{\Delta (S_T)}(-1)^{ind(x)}\cdot p_{_i}.
\]
Next, we want to sum over all critical points. First, start with the minimal orbit space with the smallest period $T_0$. Summing over all critical points of $\fT$ with minimal period $T_{0}$, we get that the contribution by $S_{T_0}$ and its iterations to $\chi\big(\C_*^{(N)}\big)$ is roughly
\[
\underset{x \ c.p. \in S_{T_0}}{\sum}\frac{N\sigma(x)}{\Delta(x)}=
\underset{x \ c.p. \in S_{T_0}}{\sum}\frac{N\sigma(S_{T})}{\Delta(S_{T})}(-1)^{ind(x)}\cdot p_{_0}=
\frac{N\sigma(S_T)}{\Delta(S_T)}\cdot\chi(S_{T_0})\cdot p_{_0}.
\]

\noindent Then, consider  any $S_{T_i}\subset S_T$. Summing over only the critical points $x\in S_{T_i}$ of $f_T$ with minimal period $T_i$, the orbifold $S_{T_i}$ and its iterations, make the following additional contribution:
\[
\underset{x \ min'l c.p.\in S_{T_i}} {\sum}\frac{N\sigma(x)}{\Delta(x)}=
\underset{x \ min'l c.p.\in S_{T_i}} {\sum}\frac{N\sigma(S_{T})}{\Delta(S_{T})}(-1)^{ind(x)}\cdot p_{_i}=
\frac{N\sigma(S_T)}{\Delta(S_T)}\cdot\CHI(S_{T_i})\cdot p_{_i}.
\]
Taking the sum over all the simple orbifolds $S_{T_i}$ of $S_T$, we get that the contribution by $S_T$ to $\chi\big(\C_*^{(N)}\big)$ is roughly
\[
\underset{S_{T_i}}{\sum}\frac{N\sigma(S_T)}{\Delta(S_T)}\cdot\CHI(S_{T_i})\cdot p_{_i}=\frac{N\sigma(S_T)}{\Delta(S_T)}\cdot\eul(S_T).
\]

Now, adding the contributions from all maximal orbit spaces with positive mean index, we get
\[
\chi\big(\C_*^{(N)}\big)=N{\sum}^+\frac{\sigma(S_T)}{\Delta(S_T)}\cdot\eul(S_T) +O(1) \text{ as } N \rightarrow \infty,
\]
and hence
\[ 
\lim_{N \to \infty}\frac{\chi\big(\C_*^{(N)}\big)}{N}={\sum}^+\frac{\sigma(S_T)}{\Delta(S_T)}\cdot\eul(S_T).
\] 

To finish the proof it remains to show that 
\begin{equation}
\label{eq:chi}
\chi^+(M,\xi)=\lim_{N\to\infty} \chi\big(\C_*^{(N)}\big)/N.
\end{equation}
By the definition of $\C_*^{(N)}$, we have 
$H_l\big(\C_*^{(N)}\big)=\HC_l(M,\xi)$ when $l_+<l<N$. Furthermore,
$|H_N\big(\C_*^{(N)}\big)-\HC_N(M,\xi)|=O(1)$ since $\dim \C_N=O(1)$.
Hence, 
\[
\chi\big(\C_*^{(N)}\big)=\sum_l(-1)^l\dim H_l\big(\C_*^{(N)}\big)
=\sum_{l=l_+}^N (-1)^l\dim\HC_l(M,\xi)+O(1)
\]
and \eqref{eq:chi} follows. This completes the proof of the theorem.

\section {Examples}\label{sec:examples}
In this section, we use the Morse-Bott version of the MEC formula to calculate the mean Euler characteristic of several examples. Some of these were previously worked out in \cite{GK}.  It should be noted that in these examples, the contact homology can also be calculated; see \cite{AM,BoPhD,Pati:thesis}. Our computations make use of some index data from in \cite{BoPhD}.

\begin{example}[The standard contact sphere]
\begin{flushleft}
Consider the  sphere $S^{2n-1}\subset\mathbb {C}^n$ with the standard contact  form 
$$\alpha_{st}=\frac{i}{2}\sum^{n}_{j=1}(z_jd\bar{z_j} - \bar{z_j}dz_j)|_{S^{2n-1}}.$$ 
The form $\alpha_{st}$ restricts to a contact form on the sphere $S^{2n-1}$ and the corresponding Reeb flow is given by 
$$( z_1, \ldots, z_n) \mapsto (e^{2\pi it}z_1,\ldots,e^{2\pi it}z_n).$$ 
All the orbits of the Reeb flow are closed forming only one simple orbit space $S_1$ and we have  
$$\MURS(S_{k1})=2kn,\ \sigma(S_1)=1, \text{ and } \ \Delta(S_1)=2n.$$
We have $$\eul(S_1)=\chi(S_1)\cdot|Stab(S_1)|=n,$$
since this orbit space is also maximal and $S_1\simeq \mathbb {CP}^{n-1}$. Putting this data into the Morse-Bott MEC formula \eqref{eq:MBMEC}, we get $$\chi^{+}(S^{2n-1},\xi_{st})=\frac{1}{2} \ \text{ and }\ \chi^{-}(S^{2n-1},\xi_{st})=0.$$
\end{flushleft}
\end{example}

\begin{example}[Circle bundle]
\label{ex. circle bundle}
Let $\pi\col M^{2n-1} \rightarrow B$ be a prequantization circle bundle over a closed symplectic manifold $(B,\omega)$. In other words, we have $\pi^* \omega = d\alpha$ where $\a$ is a connection form (real valued) on $M$. Then $\a$ is a contact form whose Reeb flow is the circle action on $M$. The Morse-Bott version of the MEC formula (Theorem \ref{thm:MBMEC}) readily applies to a calculation of $\chi^+(M,\xi)$, where $\xi=ker(\a)$, provided that $c_1(\xi)=0$ and the weak index- positivity codition is satisfied. Namely, let us assume first that $M$ is simply connected. Then 
$$\chi^+ (M,\xi)= \frac{\chi(B)}{2\left< c_1(TB),u\right>},$$
where $u\in \pi_2(B)$ is the image of a disk bounded by the fiber in $M$. (The assumption that $c_1(\xi)=0$ guarentees that the denominator is independent of the choice of the disk.) This is an immediate consequence of Theorem \ref{thm:MBMEC} and the observation that in this case $\Delta(B)=2\left< c_1(TB),u\right>$, which is routine to verify; see also \cite[Section 9.1]{BoPhD}. 

Dealing with the case when $M$ is not simply connected, fix a section $\s$ of $S^1[(\Lnc\xi)^{\otimes 2}]$, then the denominator $\Delta(B)$ can be geometrically described as follows.  For $p\in B$, pick a unitary frame of $T_{p}B$ and lift it to a trivialization along the fiber over $p$. This trivialization gives rise to a section $\s'$ of  $S^1[(\Lnc\xi)^{\otimes 2}]$ also along the fiber. Then, essentially by definition, $\Delta(B)/2$ is the rotation number of $\s'$ with respect to $\s$. (By the way, this also proves that $\Delta(B)=2\left< c_1(TB),u\right>$ in the simply connected case.)

When $M$ is simply connected, the weak index-positivity assumption is satisfied  when $(B,\omega)$ is positive monotone, i.e., $[\omega]=\lambda c_1(TB)$ on $\pi_2(M)$ and $\lambda > 0$. This condition ensures that $\Delta(B)>0$ and $\chi^{-}(M,\xi)=0$. A similar argument applies in the  weakly index-negative case, corresponding to $\lambda<0$, with the roles of $\chi^+$ and $\chi^-$ interchanged.

A similar method can also be used to calculate the mean Euler characteristic  for a general Boothby-Wang fibration; see \cite[Section 9.1.2]{Pati:thesis} for relevant index calculations. Here however, we are not going to follow this line of reasoning but instead, consider a specific example of interest.
\end{example}

\begin{example}[Ustilosky spheres]
\begin{flushleft}
A Brieskorn manifold $\Sigma(a)=\Sigma(a_0,\cdots,a_n)$ is defined as the intersection of $$\V(a)=V(a_0,\cdots,a_n)=\{(z_0,\cdots,z_n)\in\mathbb{C}^{n+1}:z_0^{a_0}+\cdots+z_n^{a_n}=0\}$$ with the unit sphere $S^{2n+1}\subset\mathbb{C}^{n+1},$ where $a_j\geq2$ are natural numbers. This manifold  admits a contact form $\alpha$. A well-known result of Brieskorn is that when $n=2m+1$ and $p=\pm1 \ (\text{mod}\ 8)$, $a_0=p, a_1=2,\ldots, a_n=2$, then $\Sigma(a)$ is diffeomorphic to $S^{4m+1}$. The 1-form 
$$\alpha_p=\frac{i}{8}\sum^{n}_{j=0}a_j(z_jd\bar{z_j} - \bar{z_j}dz_j)$$ 
restricted to $\Sigma(a)$ is a contact form. The Reeb field is $R_{\a_p}=4i(\frac{z_0}{a_0},\cdots\frac{z_n}{a_n})$ and the corresponding Reeb flow is given by
\[
(z_0,\ldots,z_n) \mapsto (e^{4it/p}z_0, e^{2it}z_1, \ldots, e^{2it}z_n).
\]

There are two simple orbifolds.
\begin{enumerate}
	\item $T=\pi$ $(z_0=0)$:
				For $S_{\pi}$, we have $|Stab(S_{\pi})|=p$, $S_{\pi} \simeq \mathbb{CP}^{n-2}$, and $\CHI(S_{\pi})=n-1$.
	\item $T=p\pi$ $(z_0\neq 0)$:
				The orbifold $S_{p\pi}$ is a maximal orbit space and contains a $p$-cover of $S_{\pi}$. Since $S_{p\pi}$ is homeomorphic to 	
				$\mathbb{CP}^{n-1}$,  we get $\CHI(S_{p\pi})=1$.  
\end{enumerate}

The index of a $k$-cover of $S_p\pi$ is given by 
$$\MURS(S_{kp\pi})=2k((n-2)p+2).$$

The above data gives us
$$\eul(S_{p\pi})=(n-1)p+1, \  \Delta(S_{p\pi})=2((n-2)p+2),\ \text{ and }\ \sigma(S_{p\pi})=1.$$

Then by the Morse-Bott MEC formula \eqref{eq:MBMEC}, we get 
$$\chi^+(M,\xi_p)=\frac{1}{2}\frac{((n-1)p+1)}{((n-2)p+2)}\ \text{ and }\ \chi^-(M,\xi_p)=0.$$

This verifies that the mean Euler characteristic distinguishes infinitely many inequivalent contact structures on $S^{2n-1}$ given by different choices of $p$ where $p=\pm1 \ (\text{mod}\ 8)$.
\end{flushleft}
\end{example}

\begin{remark}
On $S^2 \times S^3$, as is shown in \cite{AM}, there are infinitely many inequivalent contact structures $\xi_k$, for $k\in \mathbb{N}$, in the unique homotopy class determined by the vanishing of the first Chern class. It is not hard to see that the mean Euler characteristic also distinguishes the contact structures $\xi_k$. 

%In \cite{AM}, it is shown that there are infinitely many non-equivalent contact structures $\xi_k$ on $S^2 \times S^3$, for $k\in \mathbb{N}$, in the unique homotopy class determined by the vanishing of the first Chern class. It is not hard to see that the mean Euler characteristic also distinguishes the contact structures $\xi_k$. 
\end{remark}

\end{document}